\pdfoutput=1
\documentclass[reqno]{amsart}
\usepackage{amsmath,amssymb,mathrsfs,amsthm,amsfonts,mathtools}
\usepackage[inline]{enumitem} 				
\usepackage[usenames,dvipsnames]{xcolor}
\definecolor{hypercolor}{HTML}{003399}
\usepackage{hyperref}
\hypersetup{%
	colorlinks=true, linkcolor=hypercolor,
	citecolor=hypercolor
}
\usepackage{newtxtext}						
\usepackage{graphicx}
\usepackage[paper=letterpaper,margin=1in]{geometry}

\newtheorem{thm}{Theorem}[section]
\newtheorem{lem}[thm]{Lemma}
\newtheorem{prop}[thm]{Proposition}
\newtheorem{cor}[thm]{Corollary}
\newtheorem{innercustomthm}{Theorem}
  {\endinnercustomthm}
\newtheorem{innercustomprop}{Proposition}
\newenvironment{customprop}[1]
  {\renewcommand\theinnercustomprop{#1}\innercustomprop}%
  {\endinnercustomprop}
\newtheorem{innercustomlem}{Lemma}
\newenvironment{customlem}[1]
  {\renewcommand\theinnercustomlem{#1}\innercustomlem}%
  {\endinnercustomlem}
\newtheorem{innercustomcor}{Corollary}
\newenvironment{customcor}[1]
  {\renewcommand\theinnercustomcor{#1}\innercustomcor}%
  {\endinnercustomcor}
\theoremstyle{definition}

\newtheorem{rmk}[thm]{Remark}
\numberwithin{equation}{section}
\allowdisplaybreaks
\newcommand*{\Cdot}{{\raisebox{-0.5ex}{\scalebox{1.8}{$\cdot$}}}} 
\usepackage{acronym}
\acrodef{LDP}{Large Deviation Principle}
\acrodef{KPZ}{Kardar--Parisi--Zhang}
\acrodef{SHE}{Stochastic Heat Equation}
\acrodef{WNT}{Weak Noise Theory}
\acrodef{LPP}{Last Passage Percolation}

\newcommand{\e}{\varepsilon}
\newcommand{\rt}{[0,T]\times\R}				
\newcommand{\rtz}{[\cuttime,T]\times\R}		

\newcommand{\textbb}{\mathrm{bb}}			

\newcommand{\R}{\mathbb{R}}
\newcommand{\Z}{\mathbb{Z}}
\newcommand{\Zn}{\Z_{\geq 0}}       
\newcommand{\Zp}{\Z_{\geq 1}}
\newcommand{\Csp}{C}						
\newcommand{\Lsp}{L}						
\newcommand{\Hsp}{H}						
\newcommand{\bana}{\mathcal{B}}				
\newcommand{\banaE}{E}						
\newcommand{\hilb}{\mathcal{H}}				

\newcommand{\hh}{h}							
\newcommand{\dev}{\rho}						
\newcommand{\devm}{\rho_*}					
\newcommand{\hm}{\hfn_*}					
\newcommand{\scl}{\lambda}					
\newcommand{\Path}{\gamma}					%
\newcommand{\bm}{B}							
\newcommand{\bb}{B_\text{b}}				
\newcommand{\cuttime}{\eta}					
\newcommand{\hermite}{H}					
\newcommand{\ini}{g_*}						
\newcommand{\bsr}{I_*}               
\newcommand{\hk}{p}							
\newcommand{\hkk}[3]{\hk(#1,#2-#3)}			
\newcommand{\hkn}[2]{p(#1, #2)}
\newcommand{\tside}{\tau}					
\newcommand{\xside}{\zeta}					

\newcommand{\ZZ}{Z}							
\newcommand{\ZZe}{\ZZ_\e} 
\newcommand{\chaos}{Y}						
\newcommand{\pair}{W}						

\newcommand{\Zfn}{\mathsf{Z}}				
\newcommand{\chaosfn}{\mathsf{Y}}			
\newcommand{\hfn}{\mathsf{h}}				
\newcommand{\energy}{U}						
\newcommand{\rateone}{\Phi}					

\newcommand{\rate}{I}

\newcommand{\Vertices}{\mathcal{V}}			
\newcommand{\Edges}{\mathcal{E}}			
\newcommand{\edge}{\mathbf{e}}				

\renewcommand{\P}{\mathbb{P}}				
\newcommand{\E}{\mathbb{E}}					
\renewcommand{\d}{\mathrm{d}}				
\newcommand{\ind}{\mathbf{1}}				
\newcommand{\set}[1]{\{#1\}}

\newcommand{\norm}[1]{\Vert #1\Vert}		
\newcommand{\normL}[1]{\Vert #1\Vert_{L^2}}	
\newcommand{\ip}[1]{\langle #1\rangle}		
\newcommand{\normSup}[2]{\Vert #1 \Vert_{#2}}

\newcommand{\snormH}[2]{ [ #1 ]_{#2} }		

\renewcommand{\bar}{\overline}

\newcommand{\til}{\widetilde}

\title{Short time large deviations of the KPZ equation}
\author{Yier Lin and Li-Cheng Tsai}

\address[Yier Lin]{\hspace{19pt}Department of Mathematics, Columbia University}
\address[Li-Cheng Tsai]{Departments of Mathematics, Rutgers University --- New Brunswick}

\begin{document}
\begin{abstract}
We establish the Freidlin--Wentzell Large Deviation Principle (LDP) 
for the Stochastic Heat Equation with multiplicative noise in one spatial dimension.
That is, we introduce a small parameter $ \sqrt{\e} $ to the noise,
and establish an LDP for the trajectory of the solution.
Such a Freidlin--Wentzell LDP gives the short-time, one-point LDP for the KPZ equation
in terms of a variational problem.
Analyzing this variational problem under the narrow wedge initial data,
we prove a quadratic law for the near-center tail and a $ \frac52 $ law for the deep lower tail.
These power laws confirm existing physics predictions \cite{kolokolov07,kolokolov09,meerson16,ledoussal16short,kamenev16}.
\end{abstract}

\maketitle

\section{Introduction} \label{s.intro}

In this paper we study the \ac{KPZ} equation in one spatial dimension
\begin{align}
	\label{e.kpz}
	\partial_t \hh = \tfrac12 \partial_{xx} \hh + \tfrac12 (\partial_x \hh)^2 + \xi,
\end{align}
where $ \hh=\hh(t,x) $, $ (t,x)\in(0,\infty)\times\R $, and $ \xi=\xi(t,x) $ denotes the spacetime white noise.
The equation was introduced by \cite{kardar1986dynamic} 
to describe the evolution of a randomly growing interface,
and is connected to many physical systems including 
directed polymers in a random environment, last passage percolation, randomly stirred fluids, and interacting particle systems.
The equation exhibits integrability and has statistical distributions related to random matrices. 
We refer to \cite{ferrari2010random,quastel2011introduction,corwin2012kardar,quastel2015one,chandra2017stochastic,corwin2019} and the references therein for 
the mathematical study of and related to the KPZ equation.

Due to the roughness of $ \hh $, 
the term $ (\partial_x\hh)^2 $ in \eqref{e.kpz} does not make literal sense,
and the well posedness of the \ac{KPZ} equation requires renormalization \cite{hairer14,gubinelli2015paracontrolled}.
In this paper we work with the notion of Hopf--Cole solution.
Informally exponentiating $ \ZZ = \exp(\hh) $ brings the \ac{KPZ} equation to the \ac{SHE}
\begin{align}
	\label{e.she}
	\partial_t \ZZ = \tfrac12 \partial_{xx} \ZZ + \xi \ZZ.
\end{align}
It is standard to establish the well posedness of \eqref{e.she} by chaos expansion;
see Section~\ref{s.chaos.fnvalued} for more discussions on Wiener chaos.
For a function-valued initial data $ \ZZ(0,\Cdot) \geq 0 $ that is not identically zero,
\cite{mueller91} showed that $ \ZZ(t,x)>0 $ for all $ t>0 $ and $ x\in\R $ almost surely.
The Hopf--Cole solution of the \ac{KPZ} equation is then defined as $ \hh:=\log\ZZ $.
This notion of solution coincides with that of \cite{hairer14,gubinelli2015paracontrolled} under suitable assumptions.
An often considered initial data is to start the \ac{SHE} from a Dirac delta at the origin, i.e., $ \ZZ(0,\Cdot)=\delta_0(\Cdot) $,
which is referred to as the narrow wedge initial data for $ \hh $.
For such an initial data, \cite{flores14} established the positivity for $ \ZZ(t,x) $ so that the Hopf--Cole solution $ \hh:=\log\ZZ $ is well-defined.

Large deviations of the KPZ equation have been intensively studied in the mathematics and physics communities in recent years. 
Results are quite fruitful in the long time regime, $ t\to\infty $.
For the narrow wedge initial data, physics literature predicted that the one-point, lower-tail \ac{LDP} rate function
should go through a crossover from a cubic power to a $ \frac52 $ power \cite{krajenbrink2018simple}.
(The prediction of the $ \frac52 $ power actually first appeared in the short time regime; see the discussion about the short time regime below.)
The work \cite{corwin20lower} derived rigorous, detailed bounds on the one-point tail probabilities for the narrow wedge initial data
and in particular proved the cubic-to-$ \frac52 $ crossover.
Similar bounds are obtained in \cite{corwin20general} for general initial data.
The exact lower-tail rate function were derived in the physics works \cite{sasorov17,corwin18, krajenbrink18, ledoussal19}, and was rigorously proven in \cite{tsai18, cafasso19}. 
Each of these works adopts a different method.
In \cite{krajenbrink2019linear}, the four methods in \cite{sasorov17,corwin18,krajenbrink18, tsai18} were shown to be closely related.
As for the upper tail,
the physics work \cite{ledoussal16long} derived a $ \frac32 $ power law for the entire rate function under the narrow wedge initial data,
and \cite{das19} gave a rigorous proof for this upper-tail \ac{LDP}.
The work \cite{ghosal20} extended this upper-tail \ac{LDP} to general initial data.

For the finite time regime, $ t\in(0,\infty) $ fixed,
motivated by studying the positivity or regularity (of the one-point density) of the \ac{SHE} or related equations,
the works \cite{mueller91,mueller2008regularity,flores14,chen16,hu2018asymptotics} established tail probability bounds of the \ac{SHE} or related equations.

In this paper we focus on \emph{short} time large deviations of the \ac{KPZ} equation.
Employing the \ac{WNT}, the physics works \cite{kolokolov07,kolokolov09,meerson16,kamenev16} 
predicted that the one-point, lower-tail rate function should crossover from a quadratic power law to a $ \frac52 $ power law for the narrow wedge and flat initial data.
By analyzing an exact formula, the physics work \cite{ledoussal16short} 
obtained the entire one-point rate function for the narrow wedge initial data; see Section~\ref{s.Phi}.
This was confirmed by the numerical result \cite{hartmann2018high}.
From this one-point rate function \cite{ledoussal16short} also demonstrated the crossover.
The quadratic power arises from the Gaussian nature of the \ac{KPZ} equation in short time,
while the $ \frac52 $ power appears to be a persisting trait of the deep lower tail of the \ac{KPZ} equation in all time regimes.
Our main result gives the first proof of the short time \ac{LDP} for the \ac{KPZ} equation and the quadratic-to-$ \frac52 $ crossover.

\begin{thm}
\label{t.main}
Let $ h $ denote the solution of the \ac{KPZ} equation \eqref{e.kpz} with the initial data $ \ZZ(0,\Cdot)=\delta_0(\Cdot) $.
\begin{enumerate}[leftmargin=20pt, label=(\alph*)]
\item \label{t.main.exist}
For any $ \scl>0 $, the limits exist
\begin{align*}
	\lim_{t\to 0 }
	t^\frac12 \log \P\big[ \hh(2t,0) + \log\sqrt{4\pi t} \leq -\scl \big] &=: -\rateone(-\scl),
\\
	\lim_{t\to 0 }
	t^\frac12 \log \P\big[ \hh(2t,0) + \log\sqrt{4\pi t} \geq \scl \big] &=: -\rateone(\scl).
\end{align*}
\item \label{t.main.bulk}
$
	\displaystyle
	\lim_{\scl\to 0 } \scl^{-2} \rateone(\scl)
	=
	\tfrac{1}{\sqrt{2\pi}}.
$
\item \label{t.main.lower}
$
	\displaystyle
	\lim_{\scl\to\infty} \scl^{-\frac52} \rateone(-\scl) = \tfrac{4}{15\pi}.
$
\end{enumerate}
\end{thm}

\begin{rmk}
Our method works also for the flat initial data $ \hh(0,x) \equiv 0 $,
but we treat only the narrow wedge initial data to keep the paper at a reasonable length.

Our result generalizes immediately to $ h(2t,x) $, for $ x\in\R $.
This is because, under the delta initial data, the one-point law of $ Z(2t,x)/\hk(2t,x) $ does not depend on $ x $. 
This fact can be verified from the Feynman--Kac formula for the \ac{SHE}. 
\end{rmk}

\begin{rmk}
Even though \ac{LDP} rate functions are model dependent,
the $ \frac52 $ tail seems to be somewhat ubiquitous in the \ac{KPZ} class.
It shows up in all time regimes for the \ac{KPZ} equation,
and has also been observed in the TASEP \cite{derrida98}.
A very interesting question is to investigate to what extend is the $ \frac52 $ tail universal,
and to find a unifying approach to understand the origin of the tail. %
\end{rmk}

\begin{rmk}\label{r.ut}
The aforementioned physics works \cite{kolokolov09,meerson16,ledoussal16short,kamenev16} also derived the asymptotics of the deep upper tail.
The prediction is $ \lim_{\scl\to\infty} \scl^{-3/2} \rateone(\scl) = \frac{4}{3} $.
We leave this question for future work.
\end{rmk}

\begin{rmk}
The short-time large deviations for the KPZ equation were also studied under other initial data or on a half-line. 
For the KPZ equation starting from Brownian initial data, the problem was studied in physics works \cite{krajenbrink17short, meerson2017height}. 
For the half-line KPZ equation, the same problem was studied in the physics work \cite{krajenbrink2018large, meerson2018large}; see also
\cite{krajenbrink2019beyond} for a summary of these results. 
It is interesting to see whether our method generalizes in these situations.
\end{rmk}

Let us emphasize that, 
even though we follow the overarching idea of the \ac{WNT},
our method \emph{significantly differs} from existing physics heuristics. As will be explained below, the \ac{WNT} amounts to establishing a Freidlin--Wentzell LDP and analyzing the corresponding variational problem.
The second step --- analyzing the variational problem --- is the \emph{harder} step. The physics works \cite{kolokolov09,meerson16,kamenev16} provide convincing heuristic for this step by a formal PDEs argument. However, as will be explained in Section~\ref{s.intro.phys.challenge}, to make this PDE argument rigorous requires elaborate treatments and seems challenging. We hence adopt a different method.

In Section~\ref{s.intro.phys}, we will recall the physics heuristic from \cite{kolokolov09,meerson16,kamenev16}
and explain why it seems challenging to make the heuristic rigorous.
In Section~\ref{s.intro.ourmethod}, we will explain our method for proving Theorem~\ref{t.main}.

\subsection{Discussions about the physics heuristics}
\label{s.intro.phys}
Here we recall the method used in the physics works \cite{kolokolov09,meerson16,kamenev16}.
The first step is to perform scaling to turn the short-time \ac{LDP} into a Freidlin--Wentzell \ac{LDP}.
One scales
\begin{align}
	\label{e.scaling}
	\hh_\e(t,x) := \hh( \e t, \e^{1/2} x) + \log(\e^{1/2}),
\end{align}
which brings the \ac{KPZ} equation into
\begin{align}
	\label{e.kpz.scale}
	\partial_t \hh_\e = \tfrac12 \partial_{xx} \hh_\e + \tfrac12 (\partial_x \hh_\e)^2 + \sqrt{\e}\xi.
\end{align}
The term $ \log(\e^{1/2}) $ in \eqref{e.scaling} ensures that the narrow wedge initial data stays invariant.
The equation \eqref{e.kpz.scale} is in the form for studying Freidlin--Wentzell \acp{LDP}.
Roughly speaking, for a generic $ \dev\in\Lsp^2(\rt) $,
we expect $ \P[ \sqrt{\e}\xi\approx\dev ] \approx \exp(-\frac12 \e^{-1}\normL{\dev}^2) $.
When the event $ \{\sqrt{\e}\xi\approx\dev\} $ occurs, one expects $ \hh_\e $ to approximate the solution $ \hfn=\hfn(\dev;t,x) $ of
\begin{align}
	\label{e.kpz.ld}
	\partial_t \hfn = \tfrac12 \partial_{xx} \hfn + \tfrac12 (\partial_x \hfn)^2 + \dev.
\end{align} 
In more formal terms,
one expects $ \{\hh_\e\} $ to satisfy an \ac{LDP} with speed $ \e^{-1} $
and the rate function $ J(f) = \inf\{ \frac{1}{2}\normL{\dev} : \hfn(\dev) = f \} $.
Once such an \ac{LDP} is established in a suitable space,
by the contraction principle we should have
\begin{align}
	\label{e.wnt.variational.}
	\rateone(\scl) = - \lim_{\e\to 0} \e \log \P\big[ \hh_\e(2,0) \geq \scl \big] 
	&=
	\inf\big\{ \tfrac{1}{2}\normL{\dev}^2 \,:\, \hfn(\dev;2,0) \geq \scl \big\},&
	&
	\scl>0,
\\
	\label{e.wnt.variational}
	\rateone(-\scl) = - \lim_{\e\to 0} \e \log \P\big[ \hh_\e(2,0) \leq -\scl \big] 
	&=
	\inf\big\{ \tfrac{1}{2}\normL{\dev}^2 \,:\, \hfn(\dev;2,0) \leq -\scl \big\},&
	&
	-\scl<0.
\end{align}

To find the infimum in \eqref{e.wnt.variational},
one can perform variation of $ \frac12 \normL{\dev}^2 = \frac12 \int_0^2 \int_{\R} \dev^2 \, \d x\d t $ in $ \dev $ 
under the constraint $ \hfn_\scl(\dev;2,0) = -\scl $, c.f., \cite[Sect~A, Supplementary Material]{meerson16}.
The result suggests that any minimizer $ \dev $ should solve
\begin{align}
	\label{e.dev.eq}
	\partial_t \dev = -\tfrac12 \partial_{xx} \dev + \partial_x (\dev\,\partial_x\hfn).
\end{align} 
With a negative Laplacian $ -\tfrac12 \partial_{xx}\dev $, 
the equation \eqref{e.dev.eq} needs to be solved \emph{backward} in time from the terminal data $ \dev(2,x) = - c(\scl) \delta_0(x) $, c.f., \cite[Sect~A, Supplementary Material]{meerson16},
where $ c(\scl)>0 $ is a constant fixed by $ \hfn(\dev;2,0) = -\scl $.

In the near-center regime, i.e., $ \scl \to 0 $, standard perturbation arguments can be applied to analyze \eqref{e.kpz.ld} and \eqref{e.dev.eq} to conclude the quadratic power law.

We will focus on the deep lower tail regime, i.e., $ -\scl\to-\infty $.
We scale $ \scl^{-1}\hfn(\dev;t,\scl^{1/2}x) \mapsto \hfn(\dev;t,x) $ and $ \scl^{-1} \dev(t,\scl^{1/2}x) \mapsto \dev(t,x) $.
To see why such scaling is relevant,
note that, under the conditioning $ \hfn(\dev;2,0) \leq -\scl $, it is natural to scale $ \hfn $ by $ \scl^{-1} $. 
Time cannot be scaled since we are probing $ \hfn $ at $ t=2 $.
After scaling $ \hfn $ by $ \scl^{-1} $,
we find that the quadratic term $ \frac12(\partial_x\hfn)^2 $ in \eqref{e.kpz.ld} gains an excess $ \scl $ factor compared to the left hand side.
To bring the quadratic term back to the same footing as the left hand side, we scale $ x $ by $ \scl^{-1/2} $.
Similar considerations lead to the same scaling of $ \dev $. %
Under such scaling the equations \eqref{e.kpz.ld} and \eqref{e.dev.eq} become
\begin{align}
	\label{e.kpz.ld.scl}
	\partial_t \hfn &= \tfrac12 \scl^{-1} \partial_{xx} \hfn + \tfrac12 (\partial_x \hfn)^2 + \dev,
\\
	\label{e.dev.eq.scl}
	\partial_t \dev &= -\tfrac12 \scl^{-1} \partial_{xx} \dev + \partial_x (\dev \,\partial_x \hfn ).
\end{align} 
As $ \scl\to\infty $ it is tempting to drop the Laplacian terms in \eqref{e.kpz.ld.scl}--\eqref{e.dev.eq.scl}. Doing so produces
\begin{align}
	\label{e.kpz.ld.}
	\partial_t \hfn &= \tfrac12 (\partial_x \hfn)^2 + \dev,
\\	
	\label{e.dev.eq.}
	\partial_t \dev &= \partial_x (\dev \,\partial_x \hfn ),
\end{align} 
with the initial data $ \lim_{t\downarrow 0} (\hfn(t,x) t) = -\frac12 x^2 $ and the terminal data $ \dev(2,x) = -c(1)\delta_0(x) $.

The equations \eqref{e.kpz.ld.}--\eqref{e.dev.eq.} can be solved by the procedure in \cite{kolokolov09,meerson16,kamenev16}.
For the completeness of presentation we briefly recall the procedure below.
It begins by solving \eqref{e.kpz.ld.}--\eqref{e.dev.eq.} by power series expansion in $ x $. 
In view of the initial data of $ \hfn $ and the terminal data of $ \dev $,
it is natural to assume $ \hfn(t,x)=\hfn(t,-x) $ and $ \dev(t,x)=\dev(t,-x) $.
Under such assumptions, the series terminates at the quadratic power for both $ \hfn $ and $ \dev $
and produces the solution $ \hfn(t,x) = k(t) + \frac12 a(t) x^2 $ and $ \dev(t,x) = -\frac{1}{2\pi} r(t)+\tfrac{1}{2\pi} (r(t)/\ell^2(t))x^2 $.
The factor $ \frac{1}{2\pi} $ is just a convention we choose;
the functions $ a(t),k(t),r(t), $ and $ \ell(t) $ can be found by inserting the series solution in \eqref{e.kpz.ld.}--\eqref{e.dev.eq.}.
The only relevant property to our current discussion is that $ r(t)>0 $.

The series solution, however, is nonphysical.
Indeed, with $ r(t)>0 $, we have $ \normL{\dev} = \infty $.
This issue is rectified by observing that the minimizing $ \dev $ of the right hand side of \eqref{e.wnt.variational} should be nonpositive.
This is so because $ \hfn(\dev;t,x) $ increases in $ \dev $.
Hence the positive part $ \dev_+ $ of $ \dev $ would only make $ \hfn(\dev;2,0) = -1 $ harder to achieve while costing excess $ \Lsp^2 $ norm.
This observation prompts us to truncate
\begin{align*}
	\devm(t,x) := - \tfrac{1}{2\pi} r(t)\big( 1 - \tfrac{x^2}{\ell(t)^2} \big)_+.
\end{align*}
It can be verified that such a $ \devm $ and a suitably truncated $ \hfn $ solve \eqref{e.kpz.ld.}--\eqref{e.dev.eq.}.


\begin{rmk}
It may appear that the preceding scaling applies also to the upper-tail regime $ \scl\to\infty $, but that is \emph{not} the case.
In the upper-tail regime,
the analyses of the physics works \cite{kolokolov09,meerson16,kamenev16} show that,
in the \emph{pre-scaled} coordinates,
the optimal $ \dev(t,x) $ concentrates in a small corridor of size $ O(\scl^{-1/2}) $ around $ x=0 $.
This behavior is in sharp contrast with that of the lower-tail,
where the optimal $ \dev(t,x) $ spans across a region in $ x $ of width $ O(\scl^{1/2}) $ in the pre-scaled coordinate.
The distinction of behaviors in the upper- and lower-tail regimes is ubiquitous
in the KPZ universality class.
As a result, the preceding scaling does not apply to the upper-tail regime. %
\end{rmk}

\subsubsection{Challenge in making the PDE argument rigorous}
\label{s.intro.phys.challenge}
To make this PDE analysis rigorous requires elaborate treatments and seems challenging.
This is so because \eqref{e.kpz.ld.}--\eqref{e.dev.eq.} are fully nonlinear equations. Taking derivative $u = \partial_x \hfn $ in \eqref{e.kpz.ld.}--\eqref{e.dev.eq.} gives
\begin{align*}
	\partial_t u &= \tfrac{1}{2} \partial_x (u^2) + \partial_x \rho,
	\\
	\partial_t \rho &= \partial_x (\rho u).
\end{align*}
These equations do \emph{not} have unique weak solutions, just like the inviscid Burgers equation \cite[Chapter 3.4]{evans1998partial}.%
One needs to impose certain entropy conditions to ensure the uniqueness of weak solutions,
and argue that in the limit $ \scl\to\infty $ the solution of \eqref{e.kpz.ld.}--\eqref{e.dev.eq.} 
converges to the entropy solution.

\subsection{Our method}
\label{s.intro.ourmethod}
Our method, which \emph{differs} from the physics heuristic described in Section~\ref{s.intro.phys}, operates at the level of the \ac{SHE} instead of the \ac{KPZ} equation.
Recall that we defined the solution of the \ac{KPZ} equation thorough the Hopf--Cole transformation,
so the solution $ h_\e $ to \eqref{e.kpz.scale} is given by $ h_\e := \log \ZZ_\e + \log(\e^{1/2}) $, where $ \ZZ_\e $ solves %
\begin{align}
	\label{e.she.scaled}
	\partial_t \ZZ_\e = \tfrac12 \partial_{xx} \ZZ_\e + \sqrt{\e} \xi \ZZ_\e,
\end{align}
with the delta initial condition $ \ZZ_\e(0,\Cdot) = \delta_0(x) $.
We seek to establish the the Freidlin--Wentzell \ac{LDP} for \eqref{e.she.scaled}.
Roughly speaking, the \ac{LDP} states that $ \P[ \ZZ_\e \approx \Zfn ] \approx \exp(-\e^{-1} \frac12 \normL{\dev}^2) $, where $ \Zfn=\Zfn(\dev;t,x) $ solves the PDE 
\begin{align}
	\label{e.he.intro}
	\partial_t \Zfn = \tfrac12 \partial_{xx} \Zfn + \dev \Zfn.
\end{align}
The precise statement of the Freidlin--Wentzell \ac{LDP} as well as the well posedness of \eqref{e.he.intro} will be given in Section~\ref{s.intro.FW}.
Use the contraction principle to specialize the Freidlin--Wentzell \ac{LDP} to one point. We have
\begin{align}
	\label{e.intro.rate+}
	\rateone (\lambda)  &= \inf\big\{\tfrac{1}{2} \normL{\rho}^2: \log \Zfn (\dev; 2, 0) \geq \lambda \big\},
\\
	\label{e.intro.rate-}
	\rateone (-\lambda) &= \inf\big\{\tfrac{1}{2} \normL{\rho}^2: \log \Zfn (\dev; 2, 0) \leq -\lambda \big\}.
\end{align} 
To analyze the variational problems \eqref{e.intro.rate+}--\eqref{e.intro.rate-}, we express $ \Zfn $ by the Feynman--Kac formula as
\begin{align}
	\label{e.intro.feynmankac}
	\Zfn(\dev;t,x)= 
	\E_{0\to x} \Big[ \exp\Big( \int_0^t \dev(s,\bb(s)) \, \d s \Big)  \Big] \hk(t,x),
\end{align}
where the $ \E_{0\to x} $ is taken with respect to a Brownian bridge $ \bb(s) $ that starts from $ \bb(0)=0 $ and ends in $ \bb(t)=x $, and $ \hk(t,x) := \exp(-x^2/2t)/\sqrt{2\pi t} $ denotes the standard heat kernel.

Given the Feynman--Kac formula, standard perturbation argument can be applied to obtained the quadratic law in the near-center regime, $ \scl\to 0 $; this is done in Section~\ref{s.nearc}.

Here we focus on the deep lower tail regime, i.e., analyzing \eqref{e.intro.rate-} in the limit $ -\scl\to-\infty $.
The scaling $\dev(\Cdot,\Cdot) \mapsto \lambda\dev(\Cdot, \lambda^{-\frac{1}{2}} \Cdot)$  mentioned in Section~\ref{s.intro.phys} 
gives
\begin{align}\label{e.rate-.scaling}
	\rateone (-\lambda) 
	= 
	\scl^{5/2} 
	\inf\big\{\tfrac{1}{2} \normL{\dev}^2: \hfn_\scl(\dev;2,0)  \leq -1 \big\},
\end{align}
where
\begin{align}
	\label{e.intro.feynmankac.scl}
	\hfn_\scl(\dev;t,x) 
	&:= 
	(\text{lower order term}) 
	-\tfrac{x^2}{2t}
	+ 
	\scl^{-1} \log  \E_{0\to\scl^{1/2} x} \Big[ \exp\Big( \int_0^t \scl\dev(s,\scl^{-\frac12} \bb(s)) \, \d s \Big) \Big].
\end{align}
The details of this scaling are given in Section~\ref{s.feynmankac}, and the precise expression of \eqref{e.intro.feynmankac.scl} is given in \eqref{e.feynmankac.scl}.

We seek to analyze the right hand side of \eqref{e.intro.feynmankac.scl} for $ (t,x)=(2,0) $.
For a suitable class of $ \dev $, Varadhan's lemma gives, as $ -\scl\to-\infty $,
\begin{align}
	\label{e.intro.zeroT}
	\scl^{-1} \log  \E_{0\to 0} \Big[ \exp\Big( \int_0^2 \scl\dev(s,\scl^{-\frac12} \bb(s)) \, \d s \Big) \Big]
	\longrightarrow
	- \inf_{\Path} \Big\{  \int_0^2 \tfrac{1}{2} \Path'(s)^2 - \dev(s,\Path(s)) \ \d s \Big\},
\end{align}
where the infimum is taken over all $ \Hsp^1 $ path $ \Path(s) $ that starts and ends in $ 0 $, i.e., $ \Path(0)=\Path(2)=0 $. 
This limit transition is reminiscent of the convergence (under the zero-temperature limit) of the free energy of a directed polymer to that of a last passage percolation.
Our task is hence to find the $ \dev=\dev(s,y) $ with the minimal $ L^2 $ norm such that the right hand side of \eqref{e.intro.zeroT} is $ \leq -1 $.

It is natural to \emph{guess} that the minimizing $ \dev $ should be the $ \devm $ obtained in the aforementioned PDE heuristic.
Taking this explicit $ \devm $,
we prove the convergence \eqref{e.intro.zeroT} (by Varadhan's lemma) and solve the path variational problem on the right side of \eqref{e.intro.zeroT}; see Lemma~\ref{l.devm.varadhan} and Proposition~\ref{p.geodesic}.
The explicit constant $ \frac{4}{15\pi} $ in Theorem~\ref{t.main}~\ref{t.main.lower} comes from the $ L^2 $ norm of $ \devm $.

The last step is to \emph{verify} that such a $ \devm $ is indeed the minimizer. This is done in Section~\ref{s.showing>}. There we appeal to an identity \eqref{e.after.jensen.} that involves $ \devm $.
This identity follows from the fact that for $ \dev=\devm $, the right hand side of \eqref{e.intro.zeroT} is equal to $ -1 $.
Using this identity, we show that, for any $ \dev $ that satisfies the required condition $ \hfn_\scl(\dev;2,0)  \leq -1 $, the quantity $ \ip{\devm-\dev,\devm} $ is approximately $ \leq 0 $; see \eqref{e.quantitative}. This bound then verifies that $ \devm $ is the minimizer.

\subsubsection{Freidlin--Wentzell LDP for the \ac{SHE}}
\label{s.intro.FW}

Here we state our result on the Freidlin--Wentzell \ac{LDP} for the \ac{SHE} \eqref{e.she.scaled}.
For the purpose of proving Theorem~\ref{t.main}, it suffices to just consider the narrow wedge initial data,
but we also consider function-valued initial data for their independent interest.

Let us set up the notation, first for function-valued initial data.
For $ a\in\R $, define the weighted sup norm $ \normSup{g}{a} := \sup_{x \in \R} \{ e^{-a|x|} |g(x)| \} $.
Let $ \Csp_a(\R) := \{ g\in \Csp(\R) : \normSup{g}{a} <\infty\} $,
and endow this space with the norm $ \normSup{\Cdot}{a} $.
Slightly abusing notation, for functions that depend also on time,
we use the same notation 
\begin{align}
	\label{e.normSup}
	\normSup{f}{a} :=  \big \{ e^{-a|x|} |f(t,x)|  \, : \, (t,x)\in\rt \big\} 
\end{align}
to denote the analogous norm,
and let $ \Csp_a(\rt) := \{ f\in \Csp(\rt) : \normSup{f}{a} <\infty\} $, endowed with the norm $ \normSup{\Cdot}{a} $.
Adopt the notation 
$
	 \Csp_{a_*^+}(\R) := \cap_{a>a_*}\Csp_a(\R)
$
and
$
	 \Csp_{a_*^+}(\rt) := \cap_{a>a_*} \Csp_a(\rt).
$
Let $ \hk(t,x) := \exp(-\frac{x^2}{2t})/\sqrt{2\pi t} $ denote the standard heat kernel.
Recall that the mild solution of \eqref{e.she.scaled} with a deterministic initial data $ \ini $ 
is a process $ \ZZe $ that satisfies
\begin{align}
	\label{e.mild}
	\ZZe(t,x) 
	= 
	\int_\R \hkk{t}{x}{y} \ini(y) \, \d y
	+
	\e^{\frac12} \int_\R \hkk{t-s}{x}{y} \ZZe(s,y) \xi(s,y) \, \d s \d y.
\end{align}
It is standard, e.g., \cite[Sections~2.1--2.6]{quastel2011introduction}, to show that for any $ \ini\in \Csp_{a_*^+}(\R) $,
there exists a unique mild solution $ \ZZ_\e $ of \eqref{e.she.scaled} given by the chaos expansion; 
see Section~\ref{s.chaos.fnvalued} for a discussion about chaos expansion.
Further, as shown later in Corollary~\ref{c.normSup.bd}, the chaos expansion (and hence $ \ZZ_\e $) is $ \Csp_{a_*^+}(\rt) $-valued.
Next we turn to the rate function.
Fix $ \ini \in \Csp_{a_*+}(\R) $.
For $ \dev\in\Lsp^2(\rt) $, consider the PDE
\begin{align*}
	\partial_t \Zfn = \tfrac12 \partial_{xx} \Zfn + \dev \Zfn,
	\qquad
	\Zfn(\dev;0,\Cdot) = \ini(\Cdot),
\end{align*}
where $ \Zfn=\Zfn(\dev;t,x) $, $ t\in [0,T] $, and $ x\in\R $.
This PDE is interpreted in the Duhamel sense as
\begin{align}
	\label{e.he}
	\Zfn(\dev;t,x) = \int_\R \hkk{t}{x}{y} \ini(y) \, \d y + \int_0^t \int_\R \dev(s,y) \Zfn(\dev;s,y) \, \d y \d s.
\end{align}
We will show in Section~\ref{s.Zfn.fnvalued} that \eqref{e.he} admit a unique $ \Csp_{a_*^+}(\rt) $-valued solution.
We will often write $ \Zfn(\dev) = \Zfn(\dev;\Cdot,\Cdot) $
and accordingly view $ \dev\mapsto \Zfn(\dev) $ as a function $ \Lsp^2(\rt)\to\Csp_{a}(\rt) $, for $ a>a_* $.
Here $ \dev $ should be viewed as a deviation of the spacetime white noise $ \sqrt{\e}\xi $.
For each such deviation $ \dev $ we run the PDE \eqref{e.he} to obtain the corresponding deviation $ \Zfn(\dev)=\Zfn(\dev;t,x) $ of $ \ZZ_\e $.
Now, since the spacetime white noise $ \xi $ is Gaussian with the correlation $ \E[\xi(t,x)\xi(s,y)] = \delta_0(t-s)\delta_0(x-y) $,
one expects the rate function to be the $ \Lsp^2 $ norm of $ \dev $, more precisely
\begin{equation}
	\label{e.rate}
	\rate(f) := \inf\big\{\tfrac{1}{2}\normL{\dev} \,:\, \dev \in \Lsp^2(\rt),  \Zfn(\dev) = f \big\},
\end{equation}
with the convention $ \inf\emptyset := +\infty $.

As for the narrow wedge initial data,
we adopt the same notation as in the preceding but replace $ \ini\in\Csp_{a_*^+}(\R) $ with $ \ini=\delta_0 $.
More explicitly, the mild solution of the \ac{SHE} \eqref{e.she.scaled} satisfies
\begin{align}
	\tag{\ref*{e.mild}-nw}
	\label{e.mild.nw}
	\ZZe(t,x) 
	= 
	\hk(t,x)
	+
	\e^{\frac12} \int_\R \hkk{t-s}{x}{y} \ZZe(s,y) \xi(s,y) \, \d s \d y,
\end{align}
and the function $ \Zfn(\dev) $ now solves
\begin{align}
	\tag{\ref*{e.he}-nw}
	\label{e.he.nw}
	\Zfn(\dev;t,x) = \hk(t,x) + \int_0^t \int_\R \dev(s,y) \Zfn(\dev;s,y) \, \d y \d s.
\end{align}
Recall that $ \ZZe $ starts from the delta initial condition $ \ZZ_\e(0,\Cdot)=\delta_0(x) $.
The smoothing effect of the Laplacian in the \ac{SHE} makes $ \ZZ_\e(t,\Cdot) $ function-valued for all $ t>0 $,
but when $ t \to 0 $ the process $ \ZZe(t,\Cdot) $ becomes singular as it approaches $ \delta_0 $.
To avoid the singularity, we work with the space $ \Csp_a(\rtz) $, $ \cuttime>0 $, $ a\in\R $, equipped with the norm
\begin{align}
	\label{e.normSup.cut}
	\normSup{f}{a,\cuttime} :=  \big \{ e^{-a|x|} |f(t,x)|  \, : \, (t,x)\in\rtz \big\}. 
\end{align}
It is standard to show that \eqref{e.mild.nw} admits a unique solution 
that is $ \Csp_a(\rtz) $-valued for all $ \cuttime>0 $ and $ a\in\R $.
The same holds for \eqref{e.he.nw}.


Let $ \Omega $ be a topological space.
Recall that a function $ \varphi:\Omega \to \R\cup\{+\infty\} $ 
is a \textbf{good rate function} if $ \varphi $ is lower semi-continuous and the set $ \{ f: \varphi(f) \leq r \} $ is compact for all $ r<+\infty $.
Recall that a sequence $ \{W_\e\} $ of $ \Omega $-valued random variables 
\textbf{satisfies an \ac{LDP} with speed $ \e^{-1} $ and the rate function $ \varphi $} if
for any closed $ F \subset \Omega $ and open $ G\subset\Omega $,   
\begin{align*}
	\liminf_{\e \to 0} \e \log \P\big[W_\e \in G\big] \geq -\inf_{f \in G} \varphi(f),
	\qquad
	\limsup_{\e \to 0} \e \log \P\big[W_\e \in F\big] &\leq -\inf_{f \in F} \varphi(f).
\end{align*}
In this paper we prove the following Freidlin--Wentzell \ac{LDP} for the \ac{SHE}.

\begin{prop}\label{t.FW}
\begin{enumerate}[leftmargin=20pt, label=(\alph*)]
\item[]
\item \label{t.FW.fnvalued}
Fix $ a_*\in\R $, $ \ini\in\Csp_{a_*^+}(\R) $, and $ T<\infty $. Let $ \ZZe $ be the solution of \eqref{e.mild}
and let $ \Zfn(\dev) $ be the solution of \eqref{e.he}.
\\
For any $ a>a_* $, 
the function $ \rate: \Csp_a(\rt)\to\R\cup\{+\infty\} $ in \eqref{e.rate} is a good rate function.
Further, $ \{\ZZ_\e\}_\e $ satisfies an \ac{LDP} in $ \Csp_a(\rt) $ with speed $ \e^{-1} $ and the rate function $ \rate $.
\item \label{t.FW.nw}
Fix $ T<\infty $. 
Let $ \ZZe $ be the solution of \eqref{e.mild.nw} and let and let $ \Zfn(\dev) $ be the solution of \eqref{e.he.nw}.\\
For any $ a\in\R $ and $ \cuttime\in(0,T) $, 
the function $ \rate: \Csp_a(\rtz)\to\R\cup\{+\infty\} $ in \eqref{e.rate} is a good rate function.
Further, $ \{\ZZ_\e\}_\e $ satisfies an \ac{LDP} in $ \Csp_a(\rtz) $ with speed $ \e^{-1} $ and the rate function $ \rate $.
\end{enumerate}
\end{prop}

\subsection{Literature on the WNT and Freidlin-Wentzell LDPs for stochastic PDEs}
The \ac{WNT}, also known as the optimal fluctuation theory,
dates back at least to the works \cite{halperin66, zittartz66, lifshitz68} in condensed matter physics. In the context of stochastic PDEs, the \ac{WNT} studies large deviations of the solution's trajectory when the noise is scaled to be weaker and weaker. Such scaling is often equivalent to the short time scaling of a fixed SPDE. (See \eqref{e.scaling}--\eqref{e.kpz.scale} for the case of the \ac{KPZ} equation.)
In the physics literature, the \ac{WNT} was carried out in \cite{fogedby98} for the noisy Burgers equation, 
in \cite{kolokolov07, kolokolov09} for directed polymer and in \cite{kamenev16, meerson16} for the KPZ equation. 
The \ac{WNT} is also known as the instanton method in turbulence theory \cite{falkovich96, falkovich01, grafke15}, the macroscopic fluctuation theory in lattice gases \cite{bertini15}, 
and WKB methods in reaction-diffusion systems \cite{elgart04, meerson11}.

The Freidlin-Wentzell \ac{LDP} has been established for various stochastic PDEs, 
including reaction-diffusion-like stochastic equations \cite{chenal97, budhiraja08},
the stochastic Allen--Cahn equation \cite{hairer15},
and the stochastic Navier--Stokes equation \cite{cerrai19}.

\subsection{Some discussions about the rate function $ \Phi $}
\label{s.Phi}
The physics work \cite{ledoussal16short} used a different method to derive
\begin{align*}
	\Phi(\scl) 
	=
	\begin{cases}
	 \frac{-1}{\sqrt{4\pi}} \min\limits_{z \in [-1, +\infty)} \big\{ z e^{\scl} + \text{Li}_{\frac{5}{2}}(-z)\big\}, &\scl \leq \scl_c,\\
	\frac{-1}{\sqrt{4\pi}}\min\limits_{z \in [-1, 0)} \big\{z e^{\scl} + \text{Li}_{\frac{5}{2}} (-z) - \frac{8\sqrt{\pi}}{3} (-\log (-z))\big\}, & \scl \geq \scl_c,
	\end{cases}
\end{align*}
where $\text{Li}_\nu (z)$ is the poly-logarithm function and $ \scl_c = \log \zeta(\frac32) $.
Though not completely mathematically rigorous,
the derivation is based on convincing arguments 
and is backed by the numerical result \cite{hartmann2018high}. 
Based on this expression, the work obtained many properties of $ \Phi $,
including its analyticity on $ \scl\in\R $, and lower-order terms in the deep lower-tail regime $ -\scl\to-\infty $ (beyond the leading term $ \frac{4}{15\pi}\scl^{\frac52} $).
Our results do not cover these detailed properties of $ \Phi $.
Rigorously proving these properties is an interesting open question.

\subsection*{Acknowledgments}
We thank Ivan Corwin for suggesting this problem to us and for useful discussions,
thank Sayan Das, Amir Dembo, Promit Ghosal, Konstantin Matetski, and Shalin Parekh for useful discussions,
and thank Martin Hairer, Hao Shen, and Hendrik Weber for clarifying some points about the literature. We thank the anonymous referees for useful comments that improve the presentation of this paper. The research of YL is partially supported by the Fernholz Foundation's ``Summer Minerva Fellow" program and also received summer support from Ivan Corwin's
NSF grant DMS-1811143. The research of LCT is partially supported by the NSF through DMS-1953407.

\subsection*{Outline of the rest of the paper}
In Section~\ref{s.chaosZfn}, we recall the formalism of Wiener chaos,
recall a result from \cite{hairer15} that gives the \ac{LDP} for finitely many chaos,
and prepare some properties of the function $ \Zfn(\dev) $.
In Section~\ref{s.FW}, we establish tail probability bounds on the Wiener chaos for the \ac{SHE}.
Based on such tail bounds, 
we leverage the \ac{LDP} for finitely many chaos into the \ac{LDP} for the \ac{SHE}, thereby proving Proposition~\ref{t.FW}.
In Section~\ref{s.onepint}, we analyze the variational problem given by the one-point \ac{LDP} for the \ac{SHE}
and prove Theorem~\ref{t.main}.

\section{Wiener spaces, Wiener chaos, and the function $ \Zfn(\dev) $}
\label{s.chaosZfn}

In this section we recall the formalism of Wiener spaces and chaos,
and prepare some properties of $ \Zfn(\dev) $.

\subsection{Function-valued initial data}
\label{s.chaosZfn.fnvalued}
Throughout this subsection we fix $ T<\infty $, $ a_*\in\R $, and $ \ini\in\Csp_{a_*^+}(\R) $,
and initiate the \ac{SHE} \eqref{e.she.scaled} from $ \ZZe(0,\Cdot) = \ini(\Cdot) $.

\subsubsection{Wiener spaces and chaos}
\label{s.chaos.fnvalued}
We will mostly follow \cite[Section 3]{hairer15}.
The basic elements of the Wiener space formalism consists of $ (\bana,\hilb,\mu) $,
where $ \bana $ is a Banach space over $ \R $ equipped with a Gaussian measure $ \mu $,
and $ \hilb\subset\bana $ is the Cameron--Martin space of $ \bana $.
In our setting $ \hilb = \Lsp^2(\rt) $, 
and $ \bana $ can be any a Banach space such that the embedding $ \hilb\subset\bana $ is dense and Hilbert--Schmidt.
To be concrete, fixing an arbitrary orthonormal basis $ \{e_1,e_2,\ldots\} $ of $ \hilb=\Lsp^2(\rt) $, we let 
\begin{align}
	\label{e.bana}
	\bana := \Big\{ \xi=\sum \xi_i e_i \, : \, \xi_1,\xi_2,\ldots\in\R, \ \Vert \xi \Vert_{\bana} <\infty \Big\},
	\qquad
	\big\Vert \sum \xi_i e_i \big\Vert_{\bana}^2 := \sum\nolimits_{i\geq 1} \tfrac{1}{i^2}|\xi_i|^2.
\end{align}
Identifying $ \bana $ as a subset of $ \R^{\Zp} $,
we set $ \mu := \otimes_{\Zp} \nu $, where $ \nu $ is the standard Gaussian measure on $ \R $. 
The space $ \bana $ serves as the \emph{sample space}.
For example, for $ f\in\Lsp^2(\rt) $ with $ f= \sum f_i e_i $,
the function
\begin{align}
	\label{e.pair}
	\pair(f): \bana \rightarrow \R,\qquad \pair(f) := \sum\nolimits_{i\geq 1} f_i \xi_i
\end{align}
should be identified with the random variable $ \int_0^T \int_\R f(t,x) \xi(t,x) \, \d t \d x $.
This identification justifies using $ \xi $ to denote both elements of $ \bana $ and the spacetime white noise.

The Hermite polynomials $ \hermite_n(x) $ are the unique polynomials satisfying $ \deg(\hermite_n) = n $ and
\begin{equation}
	\label{e.hermiteid}
	e^{\tau x -\frac{\tau^2}{2}} = \sum_{n=0}^\infty \tau^n \hermite_n (x).
\end{equation}
The \textbf{$ n $-th $ \R $-valued Wiener chaos} is the closure in $ \Lsp^2(\bana\to\R,\mu) $ of the linear subspace spanned by
$ \prod_{i=1}^\infty H_{\alpha_i}(\pair(e_i)) $, for $ (\alpha_1,\alpha_2,\ldots) \in \Zn\times\Zn\times\ldots $ and $ \alpha_1+\alpha_2+\ldots =n $.
Since our goal is to establish a \emph{functional} \ac{LDP},
it is natural to consider Wiener chaos at the functional level.
We will follow the formalism of Banach-valued Wiener chaos from \cite[Section~3]{hairer15}.
Fix $ a>a_* $ and consider $ \banaE = \Csp_a(\rt) $, which is a separable Banach space.
The \textbf{$ n $-th $ \banaE $-valued Wiener chaos} is the space
\begin{align*}
	\Big\{ 
		\Psi\in \Lsp^2(\bana\to\banaE,\mu) \, : \, \int \Psi(\xi) \psi(\xi) \mu(\d \xi) =0, \
		\forall \psi \in (m\text{-th } \R\text{-valued Wiener chaos), with }  m\neq n
	\Big\}.
\end{align*}
In probabilistic notation, the $ n $-th $ \banaE $-valued Wiener chaos
consists of $ \Csp_a(\rt) $-valued random variables $ \Psi $ such that
$ \E[ \normSup{\Psi}{a}^2 ]<\infty $ and that $ \E[ \Psi\psi ] = 0 $,
for all $ \psi $ in the $ m $-th $ \R $-valued Wiener chaos with $ m\neq n $.

We now turn to the \ac{SHE}. Set
\begin{equation}
	\label{e.chaos}
	\chaos_{n} (t,x) 
	:= 
	\int_{\Delta_n (t)}
	\int_{\R^{n+1}} \hkk{s_{n}-s_{n+1}}{y_n}{y_{n+1}} \ini(y_{n+1}) \d y_{n+1}
	\,
	\prod_{i=1}^n \hkk{s_{i-1}-s_{i}}{y_{i-1}}{y_{i}} \xi(s_i, y_i) \,  \d s_i \d y_i, 
\end{equation} 
where
$
	\Delta_n (t)  
	= 
	\{\vec{s} = (s_0,s_1, \dots, s_{n+1}) : 0=s_{n+1} < s_n < \dots < s_1 < s_0 = t\}
$, 
with the convention $ s_0:=t $ and $ y_0 := x $.
Iterating \eqref{e.mild} gives
\begin{align}
	\label{e.chaos.expans}
	\ZZ_\e(t,x) = \sum_{n=0}^\infty \e^{\frac{n}{2}} \chaos_{n}(t,x).
\end{align}
We will show later in Proposition~\ref{p.normSup.bd} that each $ \chaos_n $ defines a $ \Csp_a(\rt) $-valued random variable,
and show in Corollary~\ref{c.normSup.bd} that the right hand side of \eqref{e.chaos.expans} converges in $ \normSup{\Cdot}{a} $ almost surely.
It is standard to show that \eqref{e.chaos.expans} gives the unique mild solution of the \ac{SHE}.
Further, given the $ n $-fold stochastic integral expression in \eqref{e.chaos},
it is standard to show that, for fixed $ (t,x)\in\rt $, the random variable $ \chaos_{n}(t,x) $ lies in the $ n $-th $ \R $-valued Wiener chaos,
and $ \chaos_{n}\in\Csp_a(\rt) =:\banaE $ lies in the $ n $-th $ \banaE $-valued Wiener chaos.
Accordingly, we refer to the series \eqref{e.chaos.expans} as the \textbf{chaos expansion} for the \ac{SHE}.

Let 
$	\ZZ_{N,\e} := \sum_{n=0}^N \e^{\frac{n}2}\chaos_{n} $
denote the partial sum of the chaos expansion \eqref{e.chaos.expans}.
The \acp{LDP} of finitely many $ \banaE $-valued Wiener chaos has been established in \cite[Theorem~3.5]{hairer15}.
We next apply this result to obtain an \ac{LDP} for $ \ZZ_{N,\e} $.
Following the notation in \cite{hairer15},
we view $ \chaos_{n} $ as a function $ \bana\to\Csp_a(\rt) $, denoted $ \chaos_{n}(\xi) $,
and define
\begin{align}
	\label{e.chaoshom}
	(\chaos_n)_\text{hom}: \Lsp^2(\rt) \to \Csp_a(\rt),
	\qquad
	(\chaos_{n})_\text{hom}(\dev) := \int_{\bana} \chaos_{n}(\xi+\dev) \, \mu(\d \xi).
\end{align}
The last integral is well-defined for any $ \dev\in\Lsp^2(\rt) $ by the Cameron--Martin theorem.
Further define
\begin{equation}
	\label{e.rateN.}
	\rate_N: \Csp_a(\rt) \to \R\cup\{+\infty\}
	\quad
	\rate_N (f) := \inf\Big\{\tfrac{1}{2} \normL{\dev}^2 \, : \, \dev \in \Lsp^2(\rt), \ \sum_{n=0}^N (\chaos_n)_\text{hom}(\dev) = f \Big\},
\end{equation}
with the convention $ \inf\emptyset := +\infty $.
We now apply \cite[Theorem~3.5]{hairer15} to obtain an \ac{LDP} for $ Z_{N,\e} $.
\begin{prop}[Special case of {\cite[Theorem~3.5]{hairer15}}]
\label{p.HW}
For any fixed $ a>a_* $,
the function $ \rate_N $ in \eqref{e.rateN.} is a good rate function.
For fixed $ N<\infty $, 
$
	\{ Z_{N,\e} := \sum_{n=0}^N \e^{\frac{n}{2}} \chaos_n \}_\e
$
satisfies an \ac{LDP} on $\Csp_a(\rt)$ with speed $ \e^{-1} $
and the rate function $ \rate_N $.
\end{prop}
\begin{proof}
Applying \cite[Theorem~3.5]{hairer15} with $ \delta(\e) = 0 $
and with $ \boldsymbol{\Psi}^{(\e)} = ( \chaos_{0}, \e^{1/2} \chaos_{1},\ldots, \e^{N/2}\chaos_N ) \in \banaE^{N+1} $ gives
an \ac{LDP} on $ \Csp_a(\rt)^{N+1} $ for $ \boldsymbol{\Psi}^{(\e)} $ with speed $ \e^{-1} $ and the rate function
$ 
	J(f_0,\ldots,f_N)
	:= 
	\inf\{ \frac{1}{2} \normL{\dev}^2 : \dev \in \Lsp^2(\rt), \ (\chaos_n)_\text{hom}(\dev) = f_n, n=0,\ldots,N \}.
$
Since the map $ \Csp_a(\rt)^{N+1} \to \Csp_a(\rt) $, $ (f_0,\ldots,f_N) \mapsto f_0+\ldots+f_N $ is continuous,
the claimed result follows by the contraction principle.
\end{proof}

\subsubsection{Properties of the function $ \Zfn(\dev) $}
\label{s.Zfn.fnvalued}
Recall that $ \Zfn(\dev) $ denotes the solution of \eqref{e.he}.
We begin by developing an series expansion for $ \Zfn(\dev) $ that mimics the chaos expansion for the \ac{SHE}.
For fixed $ \dev\in\Lsp^2(\rt) $, let
\begin{align}
	\label{e.chaosfn}
	\chaosfn_{n}(\dev;t,x)
	:=
	\int_{\Delta_n(t)}\int_{\R^{n+1}}  \hkk{s_n-s_{n+1}}{y_n}{y_{n+1}} \ini(y_{n+1}) \d y_{n+1} \, \prod_{i=1}^n \hkk{s_{i-1}-s_{i}}{y_{i-1}}{y_i} \dev (s_i,y_i) \d s_i \d y_i.
\end{align}
where
$
	\Delta_n (t)  
	:= 
	\{\vec{s} = (s_0,s_1, \dots, s_{n+1}) : 0=s_{n+1} < s_n < \dots < s_1 < s_0 = t\}
$, 
with the convention $ s_0:=t $ and $ y_0 := x $.
Iterating \eqref{e.he} shows that the unique solution is given by
\begin{align}
	\label{e.Zfn.expansion}
	\Zfn(\dev;t,x) = \sum_{n=0}^\infty \chaosfn_n(\dev;t,x),
\end{align}
provided that the right hand side of \eqref{e.Zfn.expansion} converges in $ \normSup{\Cdot}{a} $.

To verify this convergence we proceed to establish a bound on $ \normSup{\chaosfn_{n}(\dev)}{a} $.
Hereafter, we will use $ C=C(a_1,a_2,\ldots) $ to denote a deterministic positive finite constant.
The constant may change from line to line or even within the same line,
but depends only on the designated variables $ a_1,a_2,\ldots $.
Recall that $ \hk(t,x) $ denotes the standard heat kernel.
The following bounds will be useful in our subsequent analysis.
The proof of these bounds is standard and hence omitted.

\begin{lem}\label{l.hkprop}
Fix $ a\in\R $ and $ \theta\in(0,\frac{1}{2}) $. There exists $ C=C(a,\theta,T) $ such that for all $ x,x'\in\R $ and $ s<t\in[0,T] $,
\begin{enumerate}[leftmargin=20pt, label=(\alph*)]
	\item \label{item:hk0} $\hk(t,x) \leq C t^{-1/2} e^{a|x|} $,	
	\item \label{item:hk1}
	$\int_\R \hkk{t}{x}{y} e^{a|y|} \d y \leq C e^{a|x|}$,
	\item \label{item:hk2}
	$\int_\R \hkk{t}{x}{y}^2 e^{a|y|} \d y \leq C t^{-\frac{1}{2}} e^{a|y|}$,
	\item \label{item:hk3}
	$\int_\R \big(\hkk{t}{x}{y} - \hkk{t}{x'}{y}\big)^2 e^{a|y|} \d y \leq C |x-x'|^{2\theta} \, t^{-\frac{1}{2} - \theta} (e^{a|x|}\vee e^{a|x'|}) $, and
	\item \label{item:hk4}
	$\int_\R \big(\hkk{t}{x}{y} - \hkk{s}{x}{y}\big)^2 e^{a|y|} \d y \leq C |t-s|^\theta \, s^{-\frac{1}{2} - \theta} e^{a|x|} $.
\end{enumerate}

Fix $ a\in\R $, $ \cuttime\in(0,T) $, and $ \theta\in(0,\frac{1}{2}) $. There exists $ C=C(a,\theta,T,\cuttime) $ such that for all $ s<t\in[\cuttime,T] $ and $ x,x',y\in\R $,
\begin{enumerate}[leftmargin=20pt, label=(\roman*)]
	\item \label{item:hk5} $|\hkk{t}{x}{y} - \hkk{t}{x'}{y}| \leq C |x-x'|^\theta (e^{a|x-y|}\vee e^{a|x'-y|}) $, and
	\item \label{item:hk6} $|\hkn{t}{x} - \hkn{s}{x}| \leq C |t-s| e^{a|x|} $.
\end{enumerate}
\end{lem}
\noindent{}%
The next lemma gives a bound on $ \normSup{\chaosfn_{n}(\dev)}{a} $
and verifies the convergence of the right hand side of \eqref{e.Zfn.expansion}.
\begin{lem}
\label{l.chaosfn.bd}
Fix $ a>a_* $.
There exists $ C=C(T,a) $ such that,
for all $ \dev\in\Lsp^2(\rt) $ and $ n\in\Zn $, 
we have $ \normSup{\chaosfn_{n}(\dev)}{a} \leq \frac{C^n}{\Gamma(n/2)^{1/2}} \normL{\dev}^n $.
\end{lem}
\begin{proof}
Throughout this proof we write $ C=C(T,a) $.
Let $ F_n(t) := \sup_{x\in\R} e^{2a|x|}|\chaosfn_n(\dev;t,x)|^2 $.
For $ n=0 $, we have $ \chaosfn_{0}(\dev;t,x) = \int_\R \hkk{t}{x}{y} \ini(y)\d y $.
That $ \ini\in\Csp_{a^*_+}(\R) $ implies $ |\ini(y)| \leq C e^{a|y|} $. 
Combining this with Lemma~\ref{l.hkprop}\ref{item:hk1} gives $ F_0(t) \leq C $.
Next, for $ n\geq 1 $, referring to \eqref{e.chaosfn}, we see that $ \chaosfn_n(\dev;t,x) $ can be expressed iteratively as
\begin{align*}
	\chaosfn_n(\dev;t,x) = \int_0^t \int_\R \hkk{t-s}{x}{y} \chaosfn_{n-1}(\dev;s,y) \dev(s,y) \d s \d y.
\end{align*} 
Take square on both sides and apply the Cauchy--Schwarz inequality to get
$
	\chaosfn_n(\dev;t,x)^2
	\leq
	\int_0^t \int_\R \hkk{t-s}{x}{y}^2 \chaosfn_{n-1}(\dev;s,y)^2 \d s \d y \, \normL{\dev}^2.
$
Within the last integral, use $ \chaosfn_{n-1}(\dev;s,y)^2 \leq F_{n-1}(s) e^{2a|y|} $ and Lemma~\ref{l.hkprop}\ref{item:hk2},
and divide both sides by $ e^{-2a|x|} $.
We obtain $ F_n(t) \leq C \normL{\dev}^2 \int_0^t F_{n-1}(s) (t-s)^{-1/2} \d s $.
Iterating this inequality and using $ F_0(t) \leq C $ complete the proof.
\end{proof}

As it turns out, the function $ (\chaos_{n})_\text{hom}(\dev) $ in \eqref{e.chaoshom} is equal to $ \chaosfn_{n}(\dev) $ in \eqref{e.chaosfn}.
\begin{lem}
\label{l.chaosfn}
For any $ \dev\in\Lsp^2(\rt) $ and $ n\in\Zn $, we have
$
	(\chaos_n)_\mathrm{hom}(\dev)
	=
	\chaosfn_{n}(\dev).
$
\end{lem}
\begin{proof}
Recall the notation $ \pair(f) $ from \eqref{e.pair}.
Since $ \dev\in\Lsp^2(\rt) $, the Cameron--Martin theorem gives 
\begin{equation}
	\label{e.chaofn.CM}
	(\chaos_{n})_\text{hom}(\dev) 
	:= 
	\int_{\bana} \chaosfn_n(\dev+\xi) \mu(\d \xi) 
	=
	\E\big[ \exp\big( \pair(\dev) - \tfrac{1}{2} \normL{\dev}^2\big) \chaosfn_n \big].
\end{equation}
Taking $ \tau = \normL{\dev} $ and $ x =\pair(\dev/\normL{\dev}) $ in \eqref{e.hermiteid} gives 
$
	\exp( \pair(\dev) - \frac{1}{2} \normL{\dev}^2 ) 
	= 
	\sum_{m=0}^\infty 
	\normL{\dev}^m
	\hermite_m(\pair(\dev/\normL{\dev})). 
$
Invoke the well-known identity, c.f., \cite[Proposition 1.1.4]{Nua06},
\begin{equation}
	\label{e.hermiteint}
	\normL{\dev}^m
	\hermite_m(\pair(\dev/\normL{\dev})) 
	= 
	\int_{\Delta_m(T)} \int_{\R^m} \prod_{i=1}^m \dev(s_i, y_i) \xi(s_i, y_i) \d s_i \d y_i,
\end{equation}
insert the result into \eqref{e.chaofn.CM}, and exchange the sum and expectation in the result.
We have
\begin{align*}
	(\chaos_{n})_\text{hom}(\dev;t,x)
	= 
	\sum_{m = 0}^\infty	
	\E\bigg[
		\Big(
		\int_{\Delta_m (T)} 
		\int_{\R^m} \dev^{\otimes m}(\vec{s}, \vec{y}) 
		\prod_{i=1}^m \xi(s_i, y_i) \,  \d s_i \d y_i
		\Big) 
		\ 	
		\chaos_n(t,x)
	\bigg].
\end{align*} 
Within the last expression,
the random variable on the right hand side of  \eqref{e.hermiteint} belongs to the $ m $-th $ \R $-valued Wiener chaos.
Since $ \chaos_n $ belongs to the $ n $-th $ \banaE $-valued Wiener chaos,
the expectation is nonzero only when $ m=n $.
Calculating this expectation from \eqref{e.chaos} concludes the desired result.
\end{proof}

\subsection{The narrow wedge initial data}
\label{s.chaos.nw}
Throughout this subsection we fix $ 0<\cuttime<T<\infty $ and $ a\in\R $,
and initiate the \ac{SHE} \eqref{e.she.scaled} from $ \ZZe(0,\Cdot) = \delta_0(\Cdot) $.

For the Wiener space formalism,
the spaces $ \hilb=\Lsp^2(\rt) $ and $ \bana $ remain the same as in Section~\ref{s.chaos.fnvalued},
while the space $ \banaE $ now changes to $ \banaE = \Csp_a(\rtz) $.
The chaos expansion takes the same form as \eqref{e.chaos.expans} but with
\begin{equation}
\tag{\ref*{e.chaos}-nw}
\label{e.chaos.nw}
\chaos_{n} (t,x) 
:= 
\int_{\Delta_n (t)}
\int_{\R^{n+1}} \hk(s_{n}-s_{n+1},y_n)
\,
\prod_{i=1}^n \hkk{s_{i-1}-s_{i}}{y_{i-1}}{y_{i}} \xi(s_i, y_i) \,  \d s_i \d y_i.
\end{equation} 
Recall the norm $ \normSup{\Cdot}{a,\cuttime} $ from \eqref{e.normSup.cut}.
Proposition~\ref{p.normSup.bd.nw} in the following asserts that each $ \chaos_n $ defines a $ \Csp_a(\rtz) $-valued random variable,
and Corollary~\ref{c.normSup.bd.nw} asserts that the right hand side of \eqref{e.chaos.expans} converges in $ \normSup{\Cdot}{a,\cuttime} $ almost surely.
The functions $ (\chaos_{n})_\text{hom}(\dev) $ and $ \rate_N $ are defined the same way as in Section~\ref{s.chaos.fnvalued},
but with $ \Csp_a(\rtz) $ in place of $ \Csp_a(\rt) $.
More explicitly,
\begin{align}
	\tag{\ref*{e.chaoshom}-nw}
	\label{e.chaoshom.nw}
	&(\chaos_n)_\text{hom}: \Lsp^2(\rt) \to \Csp_a(\rtz),
	\qquad
	(\chaos_{n})_\text{hom}(\dev) \!:=\! \int_{\bana} \chaos_{n}(\xi+\dev) \, \mu(\d \xi),
\\
	\tag{\ref*{e.rateN.}-nw}
	\label{e.rateN..nw}
	&\rate_N: \Csp_a(\rtz) \to \R\cup\{+\infty\},
	\
	\rate_N (f) \!:=\! \inf\Big\{\tfrac{1}{2} \normL{\dev}^2  : \dev \in \Lsp^2(\rt), \ \sum_{n=0}^N (\chaos_n)_\text{hom}(\dev) = f \Big\},
\end{align}
with the convention $ \inf\emptyset := +\infty $.

Likewise, for the equation \eqref{e.he.nw}, the unique solution is given by the expansion \eqref{e.Zfn.expansion} but with
\begin{align}
	\tag{\ref*{e.chaosfn}-nw}
	\label{e.chaosfn.nw}
	\chaosfn_{n}(\dev;t,x)
	:=
	\int_{\Delta_n(t)}\int_{\R^{n}}  \hk(s_n-s_{n+1},y_n) \, \prod_{i=1}^n \hkk{s_{i-1}-s_{i}}{y_{i-1}}{y_i} \dev (s_i,y_i) \d s_i \d y_i.
\end{align}
Similar proofs of Proposition~\ref{p.HW} and Lemmas~\ref{l.chaosfn.bd} and \ref{l.chaosfn} applied in the current setting give
\begin{customprop}{\ref*{p.HW}-nw}
\label{p.HW.nw}
For any fixed $ a\in\R $ and $ \cuttime\in(0,T) $, the function $ \rate_N $ in \eqref{e.rateN..nw} is a good rate function.
For fixed $ N<\infty $,
$
	\{ Z_{N,\e} := \sum_{n=0}^N \e^{\frac{n}{2}} \chaos_n \}_\e
$
satisfies an \ac{LDP} on $\Csp_a(\rt)$ with speed $ \e^{-1} $
and the rate function $ \rate_N $.
\end{customprop}

\begin{customlem}{\ref*{l.chaosfn.bd}-nw}
\label{l.chaosfn.bd.nw}
Fix $ a\in\R $ and $ \cuttime<T\in(0,\infty) $.
There exists $ C=C(T,a,\cuttime) $ such that,
for all $ \dev\in\Lsp^2(\rt) $ and $ n\in\Zn $, 
we have $ \normSup{\chaosfn_{n}(\dev)}{a,\cuttime} \leq \frac{C^n}{\Gamma(n/2)^{1/2}} \normL{\dev}^n $.
\end{customlem}

\begin{customlem}{\ref*{l.chaosfn}-nw}
\label{l.chaosfn.nw}
For any $ \dev\in\Lsp^2(\rt) $ and $ n\in\Zn $, we have
$
	(\chaos_n)_\mathrm{hom}(\dev)
	=
	\chaosfn_{n}(\dev).
$
\end{customlem}

\section{Freidlin--Wentzell LDP for the SHE}
\label{s.FW}
%

\subsection{Function-valued initial data}
\label{s.FW.fnvalued}

Throughout this subsection, we fix $ T<\infty $, $ a_* \in \R$, and $ \ini\in\Csp_{a_*^+}(\R) = \cap_{a>a_*} \Csp_a(\R) $,
and let $ \ZZ_\e $ denote the solution of \eqref{e.she.scaled} with the initial data $ \ini $.

Recall from Proposition~\ref{p.HW} that
$
	Z_{N,\e} := \sum_{n=0}^N \e^{\frac{n}{2}} \chaos_n
$
satisfies an \ac{LDP} with the rate function $ \rate_N $ given in \eqref{e.rateN.}.
By Lemma~\ref{l.chaosfn}, the function $ \rate_N $ can be expressed as
\begin{equation}
	\label{e.rateN}
	\rate_N (f) := \eqref{e.rateN.} = \inf\Big\{\tfrac{1}{2} \normL{\dev}^2 \, : \, \dev \in \Lsp^2(\rt), \ \sum_{n=0}^N \chaosfn_n(\dev) = f \Big\}.
\end{equation}
Recall that $ \Zfn(\dev) = \sum_{n=0}^\infty \chaosfn_n(\dev) $.
Referring to the definition of $ \rate $ in \eqref{e.rate},
we see that formally taking $ N\to\infty $ in \eqref{e.rateN} produces $ \rate(f) $.
The proof of Proposition~\ref{t.FW} hence amounts to justifying this limit transition at the level of \acp{LDP}.
Key to justifying such a limit transition is 
a tight enough  bound on the tail probability $ \P[\normSup{\chaos_n}{a}\geq r] $, which we establish in Section~\ref{s.chaostail}.

\subsubsection{Tail probability of $ \normSup{\chaos_{n}}{a} $}
\label{s.chaostail}
We will utilize the fact that, for any $ (t,x)\in\rt $, 
the random variable $ \chaos_n(t,x) $ belongs to the $ n $-th $ \R $-valued Wiener chaos.
For $ X $ in the $n$-th $ \R $-valued Wiener chaos,
the hypercontractivity inequality asserts that 
higher moments of $ X $ are controlled by the second moments, c.f., \cite[Theorem 1.4.1]{Nua06},
\begin{align}
	\label{e.hyperc}
	\E \big[|X|^{p}\big] \leq p^{\frac{np}{2}} \big( \E\big[|X|^{2}\big] \big)^{\frac{p}{2}},
	\quad
	\text{for all } p \geq 2.
\end{align}
We now use this inequality to produce a tail probability bound.

\begin{lem}\label{l.hyperc}
Let $X$ be an $\R$-valued random variable in the $n$-th Wiener chaos and let $ \sigma^2:= \E[X^2] $. 
There exists a universal constant $ C \in (0,\infty) $ such that,
for all $ n\in\Zp $ and $ r\geq 0 $, 
\begin{align*}
	\P\big[ |X|\geq r \big] \leq \exp\big( - \tfrac{n}{C}\sigma^{-\frac{2}{n}} r^{\frac{2}{n}} +n \big).
\end{align*}
\end{lem}
\begin{proof}
Assume without loss of generality $\sigma = 1$. 
We seek to bound $ \E[\exp(\alpha |X|^{2/n})] $ for $\alpha> 0$.
To this end, invoke Taylor expansion to get
$
	\E[\exp(\alpha |X|^{2/n})]  
	= 
	\sum_{k=0}^{n} \frac{1}{k!} \alpha^k \E[|X|^{2k/n}] 
	+ 
	\sum_{k=n+1}^{\infty} \frac{1}{k!} \alpha^k \E[|X|^{2k/n}].
$
On the right hand side, use \eqref{e.hyperc} to bound the moments for $ k \geq n+1 $.
As for $ k\leq n $, we simply bound $ \E[|X|^{2k/n}] \leq (\E[ |X|^2 ])^{k/n} = 1 $.
Combining these bounds gives
$
	\E[\exp(\alpha |X|^{2/n})]
	\leq 
	\sum_{k = 0}^{n} \frac{1}{k!} \alpha^k  + \sum_{k=n+1}^{\infty} \frac{1}{k!} \alpha^k (\frac{2k}{n})^{k} .
$
The first term on the right hand side is bounded by $e^{\alpha}$. 
For the second term, using the inequality $ k^k \leq e^k k!$ gives 
$ \sum_{k= n+1}^{\infty} \frac{1}{k!} \alpha^k (\frac{2k}{n})^{k} \leq \sum_{k = n+1}^\infty (\frac{2e\alpha}{n})^k$. 
Combining these bounds and setting $\alpha = n/(4e) $ in the result gives
$ \E[\exp(\frac{n}{4e} |X|^{2/n})] \leq e^{\frac{n}{4e}} + 2^{-n} \leq e^n $. 
Now applying Markov's inequality completes the proof. 
\end{proof}

In light of Lemma~\ref{l.hyperc},
bounding the tail probability of $ \chaos_{n}(t,x) $ amounts to bounding its second moment,
which we do next.
Recall that $ T $, $ \ini\in\Csp_{a_*^+}(\R) $, and $ a_*\in\R $ are fixed throughout this section.

\begin{prop}\label{p.chaos.2ndmom}
Fix $ a>a_* $, $ \theta_1\in(0,1) $, $\theta_2\in(0,\frac12) $, and $ n \in \Zp $. 
There exists $ C=C(T,a,\theta_1,\theta_2) $ such that for all $ t, t'\in[0,T] $ and $ x,x'\in \R $,
\begin{enumerate}[leftmargin = 20pt, label = (\alph*)]
\item \label{item:chaosmoment onept} $\E\big[\chaos_n(t, x)^2\big] \leq  e^{2a|x|} \frac{C^n}{\Gamma(\frac{n}{2})} $,
\item \label{item:chaosmoment.x}
$
	\E\big[\big(
		\chaos_n(t, x) - \chaos_n(t, x')
	\big)^2\big] 
	\leq 
	\frac{C^n}{\Gamma(\frac{n}{2})} ( e^{2a|x|} \vee e^{2a|x'|})
	|x - x'|^{\theta_1}
$, and 
\item \label{item:chaosmoment.t}
$
	\E\big[\big(
		\chaos_n(t, x) - \chaos_n(t', x)
	\big)^2\big] 
	\leq 
	\frac{C^n}{\Gamma(\frac{n}{2})} e^{2a|x|}
	|t - t'|^{\theta_2}
$.
\end{enumerate}
\end{prop}
\begin{proof}
Fix $ a>a_* $, $ \theta_1 \in (0, 1)$, $\theta_2 \in (0, \frac{1}{2})$, and $ n \in \Zp $.
Throughout this proof we write $ C=C(T,\ini,a,\theta_1,\theta_2) $.

\ref{item:chaosmoment onept}\ %
We begin by developing an iterative bound.
It is readily verified from \eqref{e.chaos} that the chaos can be expressed as 
\begin{align}
	\label{e.chaos.ierative}
	\chaos_n (t, x) = \int_0^t \int_\R \hkk{t-s}{x}{y} \chaos_{n-1} (s, y) \xi(s, y) \d s\d y.
\end{align}
Applying It\^{o}'s isometry gives 
$	
	\E [\chaos_n (t, x)^2] = \int_0^t \int_{\R} \hkk{t-s}{x}{y}^2 \E [\chaos_{n-1} (s, y)^2 ] \d s\d y.
$
To streamline notation, set $ F_n (s) := \sup_{x \in \R}e^{-2a|x|} \E [\chaos_n(s, x)^2 ] $.
The last integral is bounded by
$	
	\int_0^t F_{n-1}(s) \int \hkk{t-s}{x}{y}^2 e^{2a|y|} \d y. 
$
Further using Lemma~\ref{l.hkprop}~\ref{item:hk2} to bound the last integral gives 
$
	\E [\chaos_n(t, x)^2] \leq 
	C \int_0^t (t-s)^{-\frac{1}{2}} e^{2a|x|} F_{n-1} (s) \d s.
$
Multiplying both sides by $\exp(-2a|x|)$ and taking the supremum over $x$ give
\begin{equation}\label{e.iterate}
	F_n (t) \leq C \int_0^t (t-s)^{-\frac{1}{2}} F_{n-1}(s)\d s.
\end{equation}
To utilize the iterative bound \eqref{e.iterate}, we need to establish a bound on $ F_0(t) $.
By definition 
\begin{align*}
F_0 (t) := \sup_{x\in\R} \Big\{ e^{-2a |x|} \Big( \int \hkk{t}{x}{y} \ini(y)\d y  \Big)^2 \Big\}.
\end{align*}
Note that $ \ini \in \Csp_{a_*^+}(\R) $ implies $ |\ini (y)| \leq C e^{a|y|}.$
Insert this bound into the definition of $ F_0(t) $,
and use Lemma~\ref{l.hkprop}~\ref{item:hk1} to bound the resulting integral (over $ y $).
The result gives $ |F_0 (t)| \leq C $.
Iterating \eqref{e.iterate} from $ n=1 $ and using $ |F_0 (t)| \leq C $ give
$ F_n (t) \leq C^n (\Gamma(n/2))^{-1} t^n$, which concludes the desired result.

\medskip

\ref{item:chaosmoment.x}\ %
Set $ x=x $ and $ x=x' $ in \eqref{e.chaos.ierative}, 
take the difference of the result,
and Apply It\^{o}'s isometry.
We have
\begin{equation}
	\label{e.chaos.iterate.x}
	\E \big[ \big( \chaos_n(t, x) - \chaos_n(t, x') \big)^2\big]
	= 
	\int_0^t \int_\R \big(\hkk{t-s}{x}{y} - \hkk{t-s}{x'}{y}\big)^2 \E \big[\chaos_{n-1}(s,y)^2\big]  \d s \d y.
\end{equation}
Use Part~\ref{item:chaosmoment onept} to bound $ \E[\chaos_{n-1}(t,x)^2] $,
and apply Lemma~\ref{l.hkprop}~\ref{item:hk3} to bound the resulting integral.
Doing so produces the desired result.

\ref{item:chaosmoment.t}\ %
Assume without loss of generality $t> t'$. 
Set $ t=t $ and $ t=t' $ in \eqref{e.chaos.ierative}, take the difference,
and apply It\^{o}'s isometry to the result.
We have
\begin{align}
	\label{e.chaos.iterate.t}
\begin{split}
	\E\big[\big(\chaos_n(t, x) - \chaos_n(t', x)\big)^2\big]
	&= 
	\int_0^{t'}  \int_\R 
	\big(\hkk{t - s}{x}{y} - \hkk{t' - s}{x}{y} \big)^2 \E\big[\chaos_{n-1}(s, y)^2\big]\d s \d y\\ 
	&\quad 
	+ 
	\int_{t'}^{t} \int_\R
	\hkk{t - s}{x}{y} \E\big[\chaos_{n-1}(s, y)^2\big] \d s\d y.
\end{split}
\end{align}
On the right hand side,
use Part~\ref{item:chaosmoment onept} to bound $ \E[\chaos_{n-1}(s,y)^2] $,
apply Lemma~\ref{l.hkprop}~\ref{item:hk4}
and Lemma~\ref{l.hkprop}~\ref{item:hk2} to bound the resulting integrals, respectively.
Doing so produces the desired result.
\end{proof}

Based on Lemmas~\ref{l.hyperc} and Proposition~\ref{p.chaos.2ndmom}, 
we now derive some pointwise H\"{o}lder bounds on $ \chaos_{n} $.
\begin{cor}\label{c.tail}
Fix $ a\in(a_*,\infty) $, $ \alpha\in(0, \frac{1}{4}) $, and $\beta \in (0, \frac{1}{2})$. 
There exists $C = C(T,a,\alpha,\beta)$ such that for all $ n\in\Zp $, $ r \geq 0 $, $ t,t'\in[0,T] $, and $ x,x'\in\R $,
\begin{enumerate}[leftmargin=20pt, label=(\alph*)]
\item \label{enu.tailx}\
$ 
	\P\Big[\,
		|\chaos_n (t, x)  - \chaos_n (t, x')| \geq |x-x'|^\beta ( e^{a|x|} \vee e^{a|x'|}) r  
	\Big] 
	\leq 
	\exp\big(- \tfrac{1}{C} n^{\frac{3}{2}} r^\frac{2}{n} + n \big)
$, and
\item[]
\item \label{enu.tailt}\
$ 
	\displaystyle 	
	\P\Big[\,
		|\chaos_n (t', x)  - \chaos_n (t, x)| \geq e^{a|x|} |t-t'|^\alpha r  
	\Big] 
	\leq 
	\exp\big(-\tfrac{1}{C} n^{\frac{3}{2}} r^\frac{2}{n} + n \big)
$.
\end{enumerate}
\end{cor}
\begin{proof}
Set $ U := (e^{-a|x|}\wedge e^{-a|x'|}) \frac{ \chaos_n (t,x)-\chaos_n (t,x') }{ |x-x'|^{\beta} } $,
$ V := (e^{-a|x|}\wedge e^{-a|x'|}) \frac{ \chaos_n (t,x)-\chaos_n (t',x) }{ |t-t'|^{\alpha} } $,
$ \sigma^2 := \E[U^2] $, and $ {\eta}^2 := \E[{V}^2] $.
Proposition~\ref{p.chaos.2ndmom}~\ref{item:chaosmoment.x} and \ref{item:chaosmoment.t} give
$ \sigma^2 \leq C^n/\Gamma(\frac{n}{2}) $ and $ {\eta}^2 \leq C^n/\Gamma(\frac{n}{2}) $.
Taking $ \frac{1}{n} $ power on both sides and using $ \Gamma(\frac{n}{2})^{-1/n} \leq C n^{-1/2} $,
we have $ \sigma^{\frac{2}{n}} \leq C n^{-1/2} $ and $ {\eta}^{\frac{2}{n}} \leq C n^{-1/2} $.
Next, since $\chaos_n (t, x) $, $ \chaos_n(t,x') $, $ \chaos_n(t',x) $, and $ \chaos_n(t',x') $ 
belong to the $ n $-th $ \R $-valued Wiener chaos, $ U $ and $ V $ also belong to the $n$-th Wiener chaos.
The desired results now follow from Lemma~\ref{l.hyperc}.
\end{proof}


Our next step is to leverage the pointwise bounds in Corollary~\ref{c.tail} to a functional bound.
To this end it is convenient to first work with H\"{o}lder seminorms. 
For $ f\in\Csp(\rt) $ and $ k\in\Z $, set
\begin{align}
	\label{e.snormHolder.}
	\snormH{f}{a,\alpha,\beta,k}
	&:= 
	e^{-a|k|}
	\sup
	\bigg\{ 
		\frac{ |f(t_1, x_1) - f(t_2, x_2)| }{ |t_1 - t_2|^\alpha + |x_1 - x_2|^\beta }
		\,:\,
		(t_1, x_1)\neq (t_2, x_2) \in [0,T]\times[k,k+1]
	\bigg\}.
\end{align}
This quantity measures the H\"{o}lder continuity of $ f $ on $ [0,T]\times[k,k+1] $.

\begin{prop}
\label{p.snormH.bd}
Fix $ a\in(a_*,\infty) $, $ \alpha\in(0,\frac14) $, and $ \beta\in(0,\frac12) $.
There exists $ C=C(T,a,\alpha,\beta) $ such that, for all $ r \geq (Cn^{-\frac12})^\frac{n}{2} $, $ n\in\Zp $, and $ k\in\Z $,
\begin{align*}
	\P\big[ \, \snormH{\chaos_{n}}{a,\alpha,\beta,k} \geq r \big]
	\leq
	C\,\exp\big( -\tfrac{1}{C} n^{\frac32} r^{\frac{2}{n}} \big). 	
\end{align*}
\end{prop}
\begin{proof}
Throughout this proof we write $ C=C(T,a_*,a,\alpha,\beta) $.

The proof follows similar argument in the proof of Kolmogorov's continuity theorem.
The starting point is an inductive partition of $ [0,T]\times [k,k+1] $ into nested rectangles.
Let $ \tside_0 := T $ and $ \xside_0:=1 $ denote the side lengths of $ R^{(0)}_{11} := [0,T]\times [k,k+1] $. 
We proceed by induction in $ \ell=0,1,2,\ldots $.
Assume, for $ \ell \geq 0 $, we have obtained the rectangles $ R^{(\ell)}_{ij} $, for $ i=1, \ldots, \prod_{\ell'=1}^{\ell-1} m_{\ell'} $ and $ j=1, \ldots, \prod_{\ell'=1}^{\ell-1} n_{\ell'} $.
We partition each $ R^{(\ell)}_{ij} $ into $ m_{\ell} \times n_{\ell} $ rectangles of equal size.
The side lengths of the resulting rectangles are therefore $ \tside_{\ell+1} = \tside_{\ell}/m_\ell $ and $ \xside_{\ell+1} = \xside_{\ell}/n_\ell $.
The numbers $ m_\ell $ and $ n_\ell $ are chosen in such a way that 
\begin{align}
	\label{e.aspect.ratio}
	& \tfrac12 \leq \tside_\ell^\alpha/\xside_\ell^\beta \leq 2,
	\quad
	\text{for } \ell = 1,2,\ldots,
\\
	\label{e.mn.bdd}
	& 2 \leq m_\ell, n_\ell \leq C,
	\quad
	\text{for } \ell = 0,1,2,\ldots.
\end{align}
Let $ \Vertices_\ell := \{ (i\tside_\ell,k+j\xside_\ell) : i=1, \ldots, \prod_{\ell'=1}^{\ell-1} m_{\ell'}, j=1, \ldots, \prod_{\ell'=1}^{\ell-1} n_{\ell'} \} $
denote the set of the vertices at the $ \ell $-th level,
and let $ \Edges_\ell $ denote the corresponding set of edges.

For $ (t_1,x_1)\neq (t_2,x_2) \in [0,T]\times[k,k+1] $, 
let 
\begin{align}
	\label{e.ell*}
	\ell_* = \ell_*(t_1,x_1,t_2,x_2) := \min \{ \ell\in\Zn \, : |t_1-t_2| \geq \tside_\ell \text{ or } |x_1-x_2| \geq \xside_\ell  \}.
\end{align}
It is standard to show that, for any $ f\in \Csp(\rt) $,
\begin{align}
	\label{e.retangle.f}
	|f(t_1,x_1)-f(t_2,x_2)|
	\leq
	C \sum_{\ell\geq\ell_*} \max_{\edge\in\Edges_\ell} |f(\partial\edge )|.
\end{align}
Here $ |f(\partial\edge )| := |f(s_1,y_1)-f(s_2,y_2)| $, where $ (s_1,y_1) $ and $ (s_2,y_2) $ are the two ends of the edge $ \edge\in\Edges_\ell $.

Below we will apply \eqref{e.retangle.f} for $ f=e^{-a|k|}\chaos_{n} $.
To prepare for this application let us first derive a bound on
\begin{align}
	\label{e.edge.union.bd}
	\sum_{\ell_0 \geq 0}
	\P\Big[\, \sum_{\ell \geq \ell_0} \max_{\edge\in\Edges_\ell} e^{-a|k|} |\chaos_{n}(\partial\edge)| \geq (\tside_{\ell_0}^\alpha+ \xside_{\ell_0}^\beta) r \Big].
\end{align}
Set $ \delta := (\frac{1}{2} (\frac14-\alpha)) \wedge (\frac12(\frac12-\beta)) $.
Fix any edge $ \edge\in\Edges_\ell $.
If $ \edge $ is in the $ t $ direction,
apply Corollary~\ref{c.tail}\ref{enu.tailt} with 
$ \{(t,x),(t',x)\} = \partial\edge $, $ \alpha\mapsto \alpha+\delta $, and $ r\mapsto \tside_\ell^{-\delta} r $.
If $ \edge $ is in the $ x $ direction,
apply Corollary~\ref{c.tail}\ref{enu.tailx} with 
$ \{(t,x),(t,x')\} = \partial\edge $, $ \beta\mapsto \beta+\delta $, and $ r\mapsto \xside_\ell^{-\delta} r $.
The result gives
\begin{align}
	\label{e.edge.t}
	\P\big[ e^{-a|k|-|a|} |\chaos_{n}(\partial\edge)| \geq \tside_\ell^\alpha r \big]
	&\leq
	\exp\big( -\tfrac1C n^\frac32 \tside^{-\delta}_\ell r^{\frac{2}{n}} + n \big),
	\quad
	\text{if } \edge \text{ is in the } t \text{ direction,}
\\
	\label{e.edge.x}
	\P\big[ e^{-a|k|-|a|} |\chaos_{n}(\partial\edge)| \geq \xside_\ell^\beta r \big]
	&\leq
	\exp\big( -\tfrac1C n^\frac32 \xside^{-\delta}_\ell r^{\frac{2}{n}} + n \big),
	\quad
	\text{if } \edge \text{ is in the } x \text{ direction.}
\end{align}
On the right hand sides of \eqref{e.edge.t}--\eqref{e.edge.x},
use $ m_\ell,n_\ell \geq 2 $ to bound $ \tside_\ell^{-\delta} \geq e^{\frac{\ell}{C}} $ and $ \xside_\ell^{-\delta} \geq e^{-\frac{\ell}{C}} $.
Take the union bound of the result over $ \edge\in\Edges_\ell $.
The condition $ m_\ell,n_\ell \leq C $ gives $ |\Edges_\ell| \leq C^\ell $.
Hence
\begin{align}
	\label{e.edge.tx}
	\P\Big[\, \max_{\edge\in\Edges_\ell} e^{-a|k|} |\chaos_{n}(\partial\edge)| \geq e^{|a|} (\tside_\ell^\alpha+ \xside_\ell^\beta) r \Big]
	\leq
	C^\ell \exp\big( - \tfrac1C e^{\frac{\ell}{C n}} n^\frac32 r^{\frac{2}{n}} +n  \big).
\end{align}
Next, the condition $ m_\ell,n_\ell \geq 2 $ 
implies $ \tside_\ell \leq \tside_{\ell_0} 2^{-\ell+\ell_0} $ and $ \xside_\ell \leq \xside_{\ell_0} 2^{-\ell+\ell_0}  $,
and therefore
$
	\sum_{\ell \geq \ell_0} (\tside_\ell^\alpha+ \xside_\ell^\beta)r
	\leq
	C (\tside_{\ell_0}^\alpha+ \xside_{\ell_0}^\beta)r.
$
Use this inequality to take the union bound of \eqref{e.edge.tx} over $ \ell \geq \ell_0 $ and absorb $ e^{|a|} $ into $ C $.
We have
\begin{align*}
	\P\Big[ 
		\sum_{\ell \geq \ell_0} \max_{\edge\in\Edges_\ell} e^{-a|k|} |\chaos_{n}(\partial\edge)| 
		\geq 
		(\tside_{\ell_0}^\alpha+\xside_{\ell_0}^\beta) Cr 
	\Big]
	\leq
	\sum_{\ell \geq \ell_0}C^\ell \exp\big( -\tfrac1C e^{\frac{\ell}{Cn}} n^\frac32 r^{\frac{2}{n}} + n \big).
\end{align*}
Use $ e^{\frac{\ell}{Cn}} \geq 1 + \frac{\ell}{Cn} $ on the right hand side,
sum both sides over $ \ell_0 \in \Zn $,
and rename $ Cr\mapsto r $.
Doing so gives\\
$
	\eqref{e.edge.union.bd}	
	\leq
	\exp( -\tfrac1Cn^\frac32 r^{\frac{2}{n}} )
	\sum_{\ell_0\geq 0} \sum_{\ell\geq\ell_0} \exp( -\tfrac{\ell}{C} n^{\frac12} r^{\frac{2}{n}} + n + \ell C).
$
For all $ r \geq (C_0n^{-\frac12})^{\frac{n}2} $ and $ C_0 $ sufficiently large,
the last double sum is convergent and bounded.
Hence
\begin{align}
	\label{e.edge.union}
	\eqref{e.edge.union.bd}
	\leq
	C \exp\big( -\tfrac1Cn^\frac32 r^{\frac{2}{n}} \big),
	\quad
	\text{for all } r \geq (Cn^{-\frac12})^{\frac{n}2}.
\end{align}

Now, set $  f= e^{-a|k|} \chaos_{n} $ in \eqref{e.retangle.f} and use \eqref{e.edge.union}.
We have that, for any $ r \geq (Cn^{-\frac12})^{\frac{n}{2}} $,
\begin{align}
	\label{e.edge.union.}
	e^{-a|k|} |\chaos_{n}(t_1,x_1)-\chaos_{n}(t_2,x_2)|
	\leq 
	C\,(\tside_{\ell_*}^\alpha+\xside_{\ell_*}^\beta) r,
	\quad
	\forall (t_1,x_1),(t_2,x_2)\in [0,T]\times[k,k+1]
\end{align}
holds with probability $ \geq 1 - C\, \exp( -\frac1C n^\frac32 r^{\frac{2}{n}} ) $.
Referring to the definition of $ \ell_* $ in \eqref{e.ell*},
we see that either $ |t_1-t_2| \geq \tside_{\ell_*} $ or $ |x_1-x_2| \geq \xside_{\ell_*} $ holds.
Combining this fact with the condition~\eqref{e.aspect.ratio} gives
$
	\frac{ \tside_{\ell_*}^\alpha+\xside_{\ell_*}^\beta }{ |t_1-t_2|^\alpha+|x_1-x_2|^\beta }
	\leq
	3.
$
Divide both sides of \eqref{e.edge.union.} by $ |t_1-t_2|^\alpha+|x_1-x_2|^\beta $,
use the last inequality on the right hand side,
take supremum of over $ (t_1,x_1)\neq (t_2,x_2)\in[0,T]\times[k,k+1] $ in the result,
and rename $ 3 C r \mapsto r $.
Doing so concludes the desired result.
\end{proof}

We now state and prove a bound on $ \P[ \,\normSup{\chaos_{n}}{a} \geq r ] $.

\begin{prop}
\label{p.normSup.bd}
Fix $ a>a_* $.
There exists $ C=C(T,a) $ such that, for all $ r \geq (Cn^{-\frac12})^{\frac{n}{2}} $ and $ n\in\Zn $,
\begin{align*}
	\P\big[ \, \normSup{\chaos_{n}}{a} \geq r \big]
	\leq
	C\,\exp\big( -\tfrac{1}{C} n^{\frac32} r^{\frac{2}{n}}\big). 	
\end{align*}
\end{prop}
\begin{proof}
Throughout this proof we write $ C=C(T,a) $.

For $ n=0 $, note that $ \chaos_{0}(t,x) = \int_{\R} \hkk{t}{x}{y} \ini(y) \, \d y $ is deterministic.
It is straightforward to check from Lemma~\ref{l.hkprop}\ref{item:hk1} and $ \ini\in\Csp_{a_*^+}(\R) $ that
$ \normSup{\chaos_{0}}{a} < \infty $.
Let $ b:=(a+a_*)/2 $.
For $ n \geq 1 $, note from \eqref{e.chaos} that $ \chaos_n(0,0) = 0 $.
Given this property,
from the definitions \eqref{e.normSup} and \eqref{e.snormHolder.} of $ \normSup{\Cdot}{a} $ and $ \snormH{\Cdot}{a,\alpha,\beta,k} $
it is straightforward to check
\begin{align*}
	\normSup{\chaos_n}{a}
	\leq
	C \sum_{k\in\Z} \snormH{\chaos_n}{a,\frac18,\frac14,k}
	\leq
	C \sum_{k\in\Z} \snormH{\chaos_n}{b,\frac18,\frac14,k} \, e^{-\frac12 (a-a_*)|k|}.
\end{align*}
Apply Proposition~\ref{p.snormH.bd} with $ r \mapsto e^{\frac12 (a-a_*)|k|} r $
and $ (a,\alpha,\beta)\mapsto(b,\frac18,\frac14) $, 
and take the union bound of the result over $ k\in\Z $.
We have
$
	\P[ \, \normSup{\chaos_n}{a} \geq C r ]
	\leq
	\sum_{k\in\Z} 
	C\,\exp( -\tfrac{1}{C} n^{\frac32} e^{\frac{|k|}{Cn}}  r^{\frac{2}{n}} ).
$
Within the last expression, use $ e^{\frac{|k|}{Cn}} \geq 1 +\frac{|k|}{Cn} $,
sum the result over $ k\in\Z $, and rename $ Cr \mapsto r $ in the result.
Doing so concludes the desired result.
\end{proof}

Proposition~\ref{p.normSup.bd} immediately implies
\begin{cor}
\label{c.normSup.bd}
Fix $ a>a_* $.
We have $ \E[ \, \normSup{\chaos_{n}}{a}^k ] < \infty $ for all $ k,n\in\Zn $,
and $ \P[ \sum_{n=0}^\infty \normSup{\chaos_{n}}{a} <\infty ] = 1 $.
\end{cor}

\subsubsection{Proof of Proposition~\ref{t.FW}~\ref{t.FW.fnvalued}}
\label{s.pf.tFW.fnvalued}

 Recall $ \rate $ from \eqref{e.rate}.
We begin by show that this function is a good rate function.

\begin{lem}\label{l.goodrate}
For any $ a>a_* $, the function $ \rate: \Csp_a(\rt) \to \R\cup\{+\infty\} $ is a good rate function.
\end{lem}
\begin{proof}
Throughout this proof we write $ \hilb = \Lsp^2(\rt) $ and $ \norm{\Cdot}_\hilb = \normL{\Cdot} $.
Recall that $ \hilb\subset\bana $ is the Cameron--Martin subspace of $ \bana $.

We begin with a reduction.
It is well-known that under $ \mu $, the random vector $ \sqrt{\e}\xi $ satisfies an \ac{LDP} on $ \bana $ with speed $ \e^{-1} $ 
and the good rate function $\bsr: \bana \to \R\cup\{+\infty\} $ given by
$
	\bsr(\dev) := 
	\frac{1}{2} \norm{\dev}^2_\hilb $ for $ \dev \in \hilb 
$
and
$
	\bsr(\dev) := + \infty
$
for 
	$ \dev \notin \hilb$,
c.f. \cite[Chapter 4]{ledoux96}.
Recall that $ \Zfn $ maps $ \hilb $ to $ \Csp_a(\rt) $.
We extend the domain of this map to $ \bana $ by setting the function be $ 0 $ outside $ \hilb $, i.e.,
\begin{align*}
	\widetilde{\Zfn}: \bana \to \Csp_a(\rt),
	\quad
	\widetilde{\Zfn}(\zeta) 
	:= 
	\left\{\begin{array}{l@{,}l}
		\Zfn(\zeta) &\text{ when } \zeta\in\hilb,
	\\
		0			&\text{ otherwise}.
	\end{array}\right.
\end{align*}
Referring to \eqref{e.rate}, 
we see that $ \rate $ is a pullback of  $\bsr$  via $ \widetilde{\Zfn} $.
Let $ \Omega(r) := \{\zeta\in \bana : \bsr(\zeta) \leq r \} $ denote a sub-level set of $ \bsr $.
By \cite[Lemma 2.1.4]{DS01}, to prove $ \rate $ 
is a good rate function,
it suffices to construct a sequence of continuous functions $ \varphi_N: \bana\to \Csp_a(\rt) $ 
such that for all $ r<\infty $, 
\begin{align}
	\tag{\ref*{e.DSgoal}'}
	\label{e.DSgoal.}
	\lim_{N\to\infty}
	\sup_{\zeta\in\Omega(r)} \normSup{ \widetilde{\Zfn}(\zeta) - \varphi_N(\zeta) }{a} =0.
\end{align}
Since $ \bsr(\zeta)<\infty $ only when $ \zeta\in\hilb $,
we have $ \Omega(r) = \{\dev\in \hilb : \norm{\dev}^2_{\hilb} \leq 2r \} $, and \eqref{e.DSgoal.} reduces to
\begin{align}
	\label{e.DSgoal}
	\lim_{N\to\infty}
	\sup_{\zeta\in\Omega(r)} \normSup{ \Zfn(\dev) - \varphi_N(\dev) }{a} =0.
\end{align}

We will construct the $ \varphi_N $ via truncation.
First, combining \eqref{e.Zfn.expansion} and Lemma~\ref{l.chaosfn} gives, for $ \dev\in\hilb $,
\begin{align}
	\label{e.Zfn.approx}
	\Zfn(\dev) = \sum_{n=0}^\infty \chaosfn_n(\dev) = \sum_{n=0}^{N} (\chaos_n)_\text{hom}(\dev) + \sum_{n> N} \chaosfn_n(\dev).
\end{align}
The $ n>N $ terms in \eqref{e.Zfn.approx} can be bounded by Lemma~\ref{l.chaosfn.bd}.

Focusing on the $ n\leq N $ terms in \eqref{e.Zfn.approx}, we seek to approximate each $ (\chaos_n)_\text{hom}(\dev) $ by a continuous function.
To this end we follow the argument in \cite[Section~3]{hairer15}.
Recall the notation $ \pair(f) $ from \eqref{e.pair}
and recall the orthonormal basis $ \{e_1,e_2,\ldots\} \subset\hilb $ from Section~\ref{s.chaos.fnvalued}.
Regarding $ \pair(e_i): \bana \to \R $ as a random variable,
we let $ \mathcal{F}_k $ be the sigma algebra generated by $ \pair(e_1), \dots, \pair(e_k) $, 
and set $ \Psi_{n,k} := \E[ \chaos_n | \mathcal{F}_k ] $.
Given that $ \chaos_n $ belongs to the $ n $-th $ \banaE $-valued Wiener chaos (recall that $ \banaE = \Csp_a(\rt) $),
it is standard to check:
\begin{enumerate}[label=(\roman*)]
\item \label{item.HW1}
	$ \lim_{k\to\infty} \E[  \normSup{ \chaos_n - \Psi_{n,k} }{a}^2 ] =0 $,
\item \label{item.HW2}
	$ \Psi_{n,k}  $ can be expressed as a finite sum of the form $ \Psi_{n,k} = \sum y_\alpha \prod_{i=1}^k \pair(e_i)^{\alpha_i} $,
	where $ y_\alpha \in \Csp_a(\rt) $
	and $ \alpha = (\alpha_1,\alpha_2,\ldots) \in \Zn\times\Zn\times \ldots $.
\end{enumerate}
Now consider the function
$
	(\Psi_{n,k})_\text{hom}: \bana \to \Csp_a(\rt)
$
defined by
$
	(\Psi_{n,k})_\text{hom}(\zeta) := \int_{\bana} \Psi_{n,k}(\xi+\zeta) \mu(\d \xi).
$
A priori, such an integral is guaranteed to be well-defined only for $ \zeta\in\hilb $.
Yet for the special case considered here,
the integral is well-defined for all $ \zeta\in\bana $ and the result gives a continuous function $ \bana \to \Csp_a(\rt) $.
To see why,
recall the definition of $ \bana $ from \eqref{e.bana},
and for $ \zeta\in\bana $ write $ \zeta = \sum_{i\geq 1} \zeta_i e_i $.
From \ref{item.HW2} we have $ \int_{\bana} \Psi_{n,k}(\xi+\zeta) \mu(\d \xi) = \sum y_\alpha \prod_{i=1}^k \E[(\zeta_i+\Xi_i)^{\alpha_i}] $, 
where $ \Xi_1,\Xi_2,\ldots $ are independent standard $ \R $-valued Gaussian random variables, and the sum is \emph{finite}.
From the last expression we see that the integral is well-defined and gives a continuous function $ \bana \to \Csp_a(\rt) $.
Next, for $ \dev\in\hilb $, by the Cameron--Martin theorem, we have
$
	\normSup{ (\chaos_n)_\text{hom}(\dev) - (\Psi_{n,k})_\text{hom}(\dev) }{a}
	=
	\normSup{ 
		\int_\bana \exp\big( \pair(\dev) - \tfrac12 \norm{\dev}^2_\hilb \big) 
		\big( \chaos_n(\xi) - \Psi_{n,k}(\xi) \big)
		\mu(\d\xi) 
	}{a}.
$
Applying the Cauchy--Schwarz inequality to the last expression gives
\begin{align}
	\label{e.CM.approx}
	\normSup{ (\chaos_n)_\text{hom}(\dev) - (\Psi_{n,k})_\text{hom}(\dev) }{a}^2
	\leq
	\exp\big( \tfrac12 \norm{\dev}^2_\hilb \big) 
	\E\big[ \normSup{ \chaos_n - \Psi_{n,k} }{a}^2 \big].
\end{align}
The right hand side converges to zero as $ k\to\infty $ by \ref{item.HW1}.
We have obtained an approximate of $ (\chaos_n)_\text{hom} $ by the continuous function $ (\Psi_{n,k})_\text{hom} $.

We now construct $ \varphi_N $.
For fixed $ N $, invoke \ref{item.HW1} to obtain $ k_n\in\Zp $ such that
$ \E[ \normSup{ \chaos_n - \Psi_{n,k_n} }{a}^2 ] \leq (N+1)^{-2} $.
Set $ \varphi_N := \sum_{n=0}^N \Psi_{n,k_n} $.
This is a continuous function $ \bana\to\Csp_a(\rt) $ since each $ \Psi_{n,k} $ is.
Subtract $ \varphi_N $ from both sides of \eqref{e.Zfn.approx},
take $ \normSup{\Cdot}{a} $ on both sides,
and use \eqref{e.CM.approx}, $ \E[ \normSup{ \chaos_n - \Psi_{n,k_n} }{a}^2 ] \leq (N+1)^{-2} $, and Lemma~\ref{l.chaosfn.bd} to bound the result.
We have, for all $ \dev\in\hilb $,
\begin{align*}
	\normSup{ \, \Zfn(\dev) - \varphi_N(\dev) }{a} 
	\leq
	\exp\big( \tfrac14 \norm{\dev}^2_\hilb \big) (N+1)^{-1}
	+
	\sum_{n\geq N} \frac{1}{\Gamma(n/2)^{\frac{1}{2}}} \big( C(a,T)\,\norm{\dev}_{\hilb}\big)^n.
\end{align*}
Now consider $ \dev\in\Omega(2r) $, whence $ \norm{\dev}^2_\hilb \leq 2r $.
We see that the desired property \eqref{e.DSgoal} follows.
\end{proof}

Recall that $ \ZZ_{N,\e} := \sum_{n=0}^N \e^{n/2}\chaos_{n} $.
Next we show that $ \ZZ_{N,\e} $ is an exponentially good approximation of $ \ZZ_\e  $.

\begin{prop}
\label{p.expapprox}
For any $ r>0 $ and $ a>a_* $, we have
$
	\lim\limits_{N \to \infty}\limsup\limits_{\e \to 0} \e \log \P\big[ \normSup{\ZZ_{N, \e} - \ZZe}{a} \geq r \big] = -\infty.
$
\end{prop}
\begin{proof}
By definition, $\ZZe - \ZZ_{N,\e} = \sum_{n >N} \e^{\frac{n}{2}}\chaos_{n} $.
Fix arbitrary $ N \in \Zp $ and $ r>0 $. 
We seek to apply Proposition~\ref{p.normSup.bd} with $ r\mapsto 2^{N-n} \e^{-n/2}r $ and $ n > N $.
For fixed $ N,r $, the required condition $ 2^{N-n}\e^{-n/2}r \geq (Cn^{-1/2})^{n/2} $ is satisfied for all $ n>N $ as long as $ \e $ is small enough.
Summing the result over $ N>n $ and applying the union bound gives
\begin{align*}
	\P\big[ \normSup{ \ZZe - \ZZ_{N,\e} }{a} \geq r \big]
	\leq
	\sum_{n>N} \P\big[ \normSup{  \chaos_n }{a} \geq 2^{N-n}\e^{-\frac{n}{2}}r \big]
	\leq
	C \, \sum_{n>N} \exp\big( - \tfrac{1}{C} \e^{-1} n^{\frac{3}{2}} e^{\frac{N-n}{Cn}} \big),
\end{align*}
where $ C=C(T,a,r) $.
On the right hand side, use $ e^{\frac{N-n}{Cn}} \geq 1 - \frac{N-n}{Cn} $ (which holds since $ n>N $), sum the result.
On both sides of the result,
apply $ \e\log(\,\Cdot\,) $, and take the limits $ \e\to 0 $ and $ N\to\infty $ in order.
Doing so concludes the desired result.
\end{proof}

We seek to apply \cite[Theorem~4.2.16~(b)]{DZ94}.
Doing so requires establishing a few properties of the rate functions.
Let $ B_r(f) := \{ f'\in\Csp_{a}(\rt) : \normSup{f'-f}{a} < r \} $
denote the open ball of radius $ r $ around $ f $.
Recall $ \rate $ from \eqref{e.rate} and recall $ \rate_N $ from \eqref{e.rateN}.
\begin{lem}\label{l.DZ}
\begin{enumerate}[leftmargin=20pt, label=(\alph*)]
\item[]
\item \label{l.DZ.1}
	For any closed $ F\subset \Csp_{a}(\rt) $, we have
	$ \displaystyle 	\inf_{f\in F} \rate(f) \leq \liminf_{N\to\infty} \, \inf_{f\in F} \rate_N(f). $
\item \label{l.DZ.2}
	For any $ f_0 \in \Csp_a(\rt) $, we have
	$ \displaystyle \rate(f_0) = \lim_{r\to 0} \liminf_{N\to\infty} \, \inf_{f\in B_r(f_0)} \rate_N(f). $
\end{enumerate}
\end{lem}
\begin{proof}
\ref{l.DZ.1}\
Let $ A $ denote the right hand side and assume without loss of generality $ A<\infty $.
Referring to the definition of $ \rate_N $ in \eqref{e.rateN},
we let $ \{(N_k,\dev_k)\}_{k=1}^\infty \subset \Zp\times\Lsp^2(\rt) $
be such that $ N_1<N_2<\ldots \to \infty $, $ \normL{\dev_k} \leq A + \frac{1}{k} $, and $ \sum_{n=0}^{N_k} \chaosfn_n(\dev_k) =: f_k \in F $.
Our next step is to relate $ (\dev_k,f_k) $ to $ \rate $.
Recall that $ \Zfn(\dev) = \sum_{n=0}^\infty \chaosfn_n(\dev) $.
Letting $ f'_k := f_k + \sum_{n>N_k} \chaosfn_n(\dev_k) \in \Csp_a(\rt) $, we have $ \Zfn(\dev_{k}) = \til{f}_k $.
Referring to the definition of $ \rate $ in \eqref{e.rate},
we see that $ \rate(f'_k) \leq \frac{1}{2}\normL{\dev_k} \leq A+\frac{1}{k} $.
Also, $ \normSup{ f'_k-f_k }{a} \leq \sum_{n>N_k} \normSup{ \chaosfn_n(\dev_k) }{a} $.
Using Lemma~\ref{l.chaosfn.bd} and $ \normL{\dev_k} \leq A+1 $ to bound the last expression gives
\begin{align}
	\label{e.fk.f'k}
	\lim_{k\to\infty}\normSup{ f'_k-f_k }{a} = 0.
\end{align}
By Lemma~\ref{l.goodrate}, the sequence $ \{f'_k\}_{k=1}^\infty $ is contained in a compact set.
Hence, after passing to a subsequence we have $ f'_k \to f_* $ in $ \Csp_a(\rt) $.
The condition \eqref{e.fk.f'k} remains true after passing to the subsequence.
Since $ f_k\in F $ and $ F $ is closed, we have $ f_*\in F $.
By Lemma~\ref{l.goodrate}, $ \rate $ is lower semi-continuous, whereby
$ \rate(f_*) \leq \liminf_{k} \rate(f'_k) $.
Lower bound the left hand side by $ \inf_{f\in F}\rate(f) $
and upper bound the right hand side by $ \liminf_k (A+\frac{1}{k}) = A $.
We conclude the desired result.

\ref{l.DZ.2}\
Apply Part~\ref{l.DZ.1}
with $ F=\bar{B_r(f_0)} $ and use the lower semicontinuity of $ \rate $ on the left hand side of the result.
Doing so gives the inequality $ \leq $ for the desired result.
It hence suffices to show the reverse inequality $ \geq $.
To this end, we assume without loss of generality $ \rate(f_0)<\infty $,
and let $ \{\til\dev_k\}_{k=1}^\infty \subset \Lsp^2(\rt) $
be such that $ \normL{\til\dev_k} \leq \rate(f_0) + \frac{1}{k} $ and that 
$ \Zfn(\dev_k) = \sum_{n=0}^{\infty} \chaosfn_n(\til\dev_k) = f_0 $.
Let $ \til{f}_k := \sum_{n=0}^{n} \Zfn(\dev_{k}) $.
Referring to the definition of $ \rate_N $ in \eqref{e.rateN},
we see that $ \rate_N(\til{f}_k) \leq \frac{1}{2}\normL{\dev_k} \leq \rate(f_0)+\frac{1}{k} $.
Also, using Lemma~\ref{l.chaosfn.bd} and $ \normL{\dev_k} \leq \rate(f_0)+1 $ gives
$
	\lim_{k\to\infty}\normSup{ f_0-\til{f}_k }{a} 
	=0.
$
This statement implies that, for any given $ r>0 $ and for all $ k $ large enough (depending on $ r $), we have $ \til{f}_k \in B_r(f_0) $.
From this and $ \rate_N(\til{f}_k) \leq \rate(f_0)+\frac{1}{k} $ the desired result follows.
\end{proof}

We are now ready to complete the proof of Proposition~\ref{t.FW}~\ref{t.FW.fnvalued}.
The \ac{LDP} for $ \{ \ZZ_{N,\e} \}_\e $ is established in Proposition~\ref{p.HW} with the rate function $ \rate_N $.
Given this, we apply \cite[Theorem~4.2.16~(b)]{DZ94} to go from the large deviations of $ \{ \ZZ_{N,\e} \}_\e $ to that of $ \{ \ZZ_\e \}_\e $.
This theorem asserts that $ \{\ZZ_\e\}_\e $ satisfies an \ac{LDP} with the rate function $ I $
contingent upon the following conditions.
\begin{enumerate}[leftmargin=20pt]
\item \label{enu.check1} $ I $ is a good rate function,
\item \label{enu.check2} $ \{ \ZZ_{N,\e} \}_\e $ is an exponentially good approximation (defined in \cite[Definition~4.2.14]{DZ94}) of $ \{ \ZZ_{\e} \}_\e $,
\item \label{enu.check3} $ \displaystyle I(f_0) = \sup_{r > 0} \liminf_{N \to \infty} \inf_{f \in B_r(f_0)} I_N (f) $, and
\item \label{enu.check4} $ \displaystyle \inf_{f \in F} I(f) \leq \limsup_{N \to \infty} \inf_{f \in F} I_m (f)$, for every closed set  $ F \subset \Csp_{a}(\rt) $.
\end{enumerate}
These conditions are verified by 
Lemma~\ref{l.goodrate},
Proposition~\ref{p.expapprox},
Lemma~\ref{l.DZ}~\ref{l.DZ.2}, and
Lemma~\ref{l.DZ}~\ref{l.DZ.1}, respectively. %
Applying \cite[Theorem~4.2.16~(b)]{DZ94} completes the proof of Proposition~\ref{t.FW}~\ref{t.FW.fnvalued}.

\subsection{The narrow wedge initial data, Proof of Proposition~\ref{t.FW}~\ref{t.FW.nw}}
\label{s.pf.tFW.nw}

Throughout this subsection, we fix $ 0<\cuttime<T<\infty $, $ a\in\R $, 
and let $ \ZZ_\e $ denote the solution of \eqref{e.she.scaled} with the initial data $ \ZZe(0,\Cdot)=\delta_0(\Cdot) $.

The proof of Proposition~\ref{t.FW}~\ref{t.FW.nw} parallels that of Proposition~\ref{t.FW}~\ref{t.FW.fnvalued}, 
starting with the analog of Proposition~\ref{p.chaos.2ndmom.nw}:

\begin{customprop}{\ref*{p.chaos.2ndmom}-nw}
\label{p.chaos.2ndmom.nw}
Fix $ \theta_1 \in (0,\frac12) $, $\theta_2\in(0,1) $, and $ n \in \Zp $. 
There exists $ C=C(T,\cuttime,a,\theta_1,\theta_2) $ such that for all $ t, t'\in[\cuttime,T] $ and $ x,x'\in \R $,
\begin{enumerate}[leftmargin = 20pt, label = (\alph*)]
\item \label{item:chaosmoment.x.nw}
$
	\E\big[\big(
		\chaos_n(t, x) - \chaos_n(t, x')
	\big)^2\big] 
	\leq 
	\frac{C^n}{\Gamma(\frac{n}{2})} (e^{2a|x|} \vee e^{2a|x'|})
	|x - x'|^{\theta_2}
$, and 
\item \label{item:chaosmoment.t.nw}
$
	\E\big[\big(
		\chaos_n(t, x) - \chaos_n(t', x)
	\big)^2\big] 
	\leq 
	\frac{C^n}{\Gamma(\frac{n}{2})} e^{2a|x|}
	|t - t'|^{\theta_1}
$.
\end{enumerate}
\end{customprop}
\begin{proof}
Throughout this proof we write $ C=C(T,\cuttime,a,\theta_1,\theta_2) $.

\ref{item:chaosmoment.x.nw}\
By \cite[Lemma~2.4]{Cor18}, we have
\begin{align}
	\label{e.Cor}
	\E[\chaos_n(t,x)^2] = t^{\frac{n}{2}} 2^{-n} \Gamma(\tfrac{n}{2})^{-1} \hkn{t}{x}^2.
\end{align}
The identity \eqref{e.chaos.iterate.x} continues to hold here.
Inserting \eqref{e.Cor} into the right hand side of \eqref{e.chaos.iterate.x} gives
\begin{align*}
	\E\big[(\chaos_n (t,x) - \chaos_n(t,x'))^2\big] 
	\leq
	\frac{C^n}{\Gamma(\frac{n}{2})} \int_0^t \int_\R \big(\hkk{t-s}{x}{y} - \hkk{t-s}{x'}{y}\big)^2 \hkn{s}{y}^2 \d y \d s. 
\end{align*}
On the right hand side, divide the integral into two parts for $ s > \cuttime/2 $ and for $ s < \cuttime/2 $.
For the former use Lemma~\ref{l.hkprop}~\ref{item:hk0} to bound $ \hk(s,y)^2 \leq C e^{2a|y|} $ (note that $ s>\cuttime/2 $)
and use Lemma~\ref{l.hkprop}~\ref{item:hk3} to bound the remaining integral;
for the latter use Lemma~\ref{l.hkprop}~\ref{item:hk5} 
to bound $ (\hkk{t-s}{x}{y} - \hkk{t-s}{x'}{y})^2 \leq C |x-x'|^{\theta_2} (e^{2a|x-y|} \vee a^{2a|x'-y|}) $
(note that $ t-s \geq \cuttime/2 $) 
and use Lemma~\ref{l.hkprop}~\ref{item:hk2} to bound the remaining integral.
Doing so concludes the desired result.

\ref{item:chaosmoment.t.nw}\
The identity \eqref{e.chaos.iterate.t} continues to hold here.
Inserting \eqref{e.Cor} into the right hand side of \eqref{e.chaos.iterate.t} gives
\begin{align}
	\E\big[(\chaos_n(t,x) - \chaos_n(t',x))^2 \big] 
	&\leq
	\frac{C^n}{\Gamma(\frac{n}{2})} 
	\Big( 
	\int_0^{t'} \int_\R \big(\hkk{t-s}{x}{y} - \hkk{t'-s}{x}{y} \big)^2 \hkn{s}{y}^2  \d y \d s
	\label{e.chaosmoment.nw.1}
\\
	\label{e.chaosmoment.nw.2}
		&\quad + \int_{t'}^{t} \int_\R \hkk{t-s}{x}{y}^2 \hkn{s}{y}^2 \d y \d s
	\Big).
\end{align}
On the right hand side of \eqref{e.chaosmoment.nw.1}, divide the integral into two parts for $ s > \cuttime/2 $ and for $ s < \cuttime/2 $.
For the former use Lemma~\ref{l.hkprop}~\ref{item:hk0} to bound $ \hkn{s}{y}^2 \leq Ce^{2a|y|} $ (note that $ s>\cuttime/2 $)
and use Lemma~\ref{l.hkprop}~\ref{item:hk4} to bound the remaining integral;
for the latter use Lemma~\ref{l.hkprop}~\ref{item:hk6} 
to bound $ (\hkk{t-s}{x}{y} - \hkk{t'-s}{x}{y})^2 \leq C |t'-t|^{\theta_1} e^{2a|x-y|} $ (note that $ t'-s \geq \cuttime/2 $)
and use Lemma~\ref{l.hkprop}~\ref{item:hk2} to bound the remaining integral.
The integral in \eqref{e.chaosmoment.nw.2} can be evaluated to be 
$
	\int_{t'}^t 4^{-1} \pi^{-3/2} t^{-1/2}s^{-1/2}(t-s)^{-1/2} \exp(-\frac{x^2}{2t}) \d s.
$
Using $ s,t\geq \cuttime $ to bound the last integral gives
$ \eqref{e.chaosmoment.nw.2} \leq C |t-t'|^{1/2} e^{2a|x|} \leq C |t-t'|^{\theta_1} e^{2a|x|} $.
From the preceding bounds we conclude the desired result.
\end{proof}

Given Proposition~\ref{p.chaos.2ndmom.nw},
a similar proof of Proposition~\ref{p.normSup.bd.nw} adapted to the current setting yields

\begin{customprop}{\ref*{p.normSup.bd}-nw}
\label{p.normSup.bd.nw}
There exists $ C=C(T,\cuttime,a) $ such that, for all $ r \geq (Cn^{-\frac12})^{\frac{n}{2}} $ and $ n\in\Zn $,
\begin{align*}
	\P\big[ \, \normSup{\chaos_{n}}{a,\cuttime} \geq r \big]
	\leq
	C\,\exp\big( -\tfrac{1}{C} n^{\frac32} r^{\frac{2}{n}}\big). 	
\end{align*}
\end{customprop}
\begin{customcor}{\ref*{c.normSup.bd}-nw}
\label{c.normSup.bd.nw}
We have $ \E[ \, \normSup{\chaos_{n}}{a,\cuttime}^k ] < \infty $ for all $ k,n\in\Zn $,
and $ \P[ \sum_{n=0}^\infty \normSup{\chaos_{n}}{a,\cuttime} <\infty ] = 1 $.
\end{customcor}

Given Proposition~\ref{p.normSup.bd.nw},
the rest of the proof for Proposition \ref{t.FW} \ref{t.FW.nw} follows the arguments in Sections~\ref{s.pf.tFW.fnvalued} mutatis mutandis.

\section{The quadratic and $ \frac{5}{2} $ laws}
\label{s.onepint}
Fix $ \ZZ_\e(0,\Cdot) = \delta_0(\Cdot) $.
Our goal is to prove Theorem~\ref{t.main}.
By the scaling \eqref{e.scaling},
we have 
\begin{align*}
	\P\big[ \hh(2\e,0) + \sqrt{4\pi\e} \geq \scl \big] = \P\big[ \sqrt{4\pi}\ZZe(2,0) \geq e^{\scl} \big],
	\quad
	\P\big[ \hh(2\e,0) + \sqrt{4\pi\e} \leq -\scl \big] = \P\big[ \sqrt{4\pi}\ZZe(2,0) \leq e^{-\scl} \big].
\end{align*}
Hence Theorem~\ref{t.main}~\ref{t.main.exist} follows from Proposition~\ref{t.FW}~\ref{t.FW.nw} (for any $ a\in\R $ and $ T\geq 2 $) and the contraction principle,
with
\begin{align}
	\label{e.rateone+}
	\rateone(\scl) &= \inf\{ \tfrac{1}{2} \normL{\dev}^2 \, : \, \sqrt{4\pi}\Zfn(\dev;2,0) \geq e^{\scl} \},
\\
	\label{e.rateone-}
	\rateone(-\scl) &= \inf\{ \tfrac{1}{2} \normL{\dev}^2 \, : \, \sqrt{4\pi}\Zfn(\dev;2,0) \leq e^{-\scl} \}.
\end{align}
Proving Theorem~\ref{t.main}~\ref{t.main.bulk} and \ref{t.main.lower} thus amounts to 
evaluating the infimums in \eqref{e.rateone+} and \eqref{e.rateone-},
which will be carried out in Sections~\ref{s.nearc} and \ref{s.lower}, respectively.

\subsection{Near-center tails, proof of Theorem~\ref{t.main}~\ref{t.main.bulk}}
\label{s.nearc}

In view of \eqref{e.rateone+} -- \eqref{e.rateone-}, our goal is to show
\begin{align}
	\label{e.nearc.goal.1}
	\lim_{\scl\to 0} 
	\scl^{-2} \inf\{ \tfrac{1}{2} \normL{\dev}^2 \, : \, \sqrt{4\pi}\Zfn(\dev;2,0) \geq e^{\scl} \} &= \tfrac{1}{\sqrt{2\pi}},
\\
	\label{e.nearc.goal.2}
	\lim_{\scl\to 0} 
	\scl^{-2} \inf\{ \tfrac{1}{2} \normL{\dev}^2 \, : \, \sqrt{4\pi}\Zfn(\dev;2,0) \leq e^{-\scl} \} &= \tfrac{1}{\sqrt{2\pi}}.
\end{align}
The proofs of \eqref{e.nearc.goal.1} and \eqref{e.nearc.goal.2} are the same so we consider only \eqref{e.nearc.goal.1}.
Fix $ \dev \in\Lsp^2([0,2]\times \R) $.
Since our goal is to prove \eqref{e.nearc.goal.1}, we assume $ \normL{\dev} \leq \scl $ and $ \scl \leq 1 $.
Recall that $ \Zfn(\dev;t,x) = \sum_{n=0}^\infty \chaosfn_n(\dev;t,x) $,
with $ \chaosfn_n(\dev;t,x) $ is given \eqref{e.chaosfn.nw}.
Let $ O(\scl^k) $ denote a generic function of $ \scl $ such that $ |O(\scl^k)|\leq C\scl^k $, for all $ \scl\in(0,1] $.
Specialize at $ (t,x) = (2,0) $ and apply the bound in Lemma~\ref{l.chaosfn.bd.nw} for $ n \geq 2 $.
We have
\begin{align}
	\label{e.nearc.1}
	\sqrt{4\pi} \Zfn(\dev;2,0)
	=
	1 + \sqrt{4\pi} \int_0^2 \int_{\R} \dev(s,y) \hkn{2-s}{y}  \hkn{s}{y} \, \d y \d s + O(\scl^2).
\end{align}

Now assume $ \sqrt{4\pi}\Zfn(\dev;2,0) \geq e^{\scl} $.
Inserting this inequality into \eqref{e.nearc.1} and Taylor expanding $ e^{\scl} $ gives
$$
	\sqrt{4\pi} \int_0^2 \int_{\R} \dev(s,y) \hkn{2-s}{y} \hkn{s}{y} \, \d y \d s \geq \scl + O(\scl^2).
$$
On the left hand side, apply the Cauchy--Schwarz inequality 
to separate $ \dev(s,y) $ and $ \hkn{2-s}{y}  \hkn{s}{y} $,
and use 
\begin{align}
	\label{e.hk.id}
	\int_0^2 \int_\R \hkn{2-s}{y}^2 \hkn{s}{y}^2 \d y \d s = 2^{-5/2}\pi^{-1/2}
\end{align}
We have $ \normL{\dev}  \geq (2/\pi)^{1/4}\scl + O(\scl^2) $.
Taking square of both sides and divide the result by $ \frac{1}{2\scl^2} $ gives the inequality `$ \geq $' in \eqref{e.nearc.goal.1}.

To show the reverse inequality, take $ \kappa>1 $ and $ \dev(s,y) = \scl\kappa 2^{3/2} \hkn{2-s}{y} \hkn{s}{y}  $.
Inserting this $ \dev $ into \eqref{e.nearc.1} and using \eqref{e.hk.id} give 
$ \sqrt{4\pi} \Zfn(\dev;2,0) \geq 1 + \kappa \scl + O(\scl^2) $.
With $ \kappa>1 $, the last expression is larger than $ e^{\scl} $ for all $ \scl $ small enough.
On the other hand, by using \eqref{e.hk.id} we have $ \frac{1}{2}\scl^{-2}\normL{\dev}^2 = \frac{\kappa^2}{\sqrt{2\pi}} $.
Hence the left hand side of \eqref{e.nearc.goal.1} is bounded by $ \frac{\kappa^2}{\sqrt{2\pi}} $.
Now taking $ \kappa \downarrow 1 $ completes the proof.

\subsection{Deep lower tail, proof of Theorem~\ref{t.main}~\ref{t.main.lower}}
\label{s.lower}

\subsubsection{The Feynman--Kac formula and scaling}
\label{s.feynmankac}

Here we consider the deep lower-tail regime, i.e., $ -\scl\to-\infty $.
The first step is to express $ \Zfn(\dev;t,x) $ by the Feynman--Kac formula.
Namely,
\begin{align}
	\label{e.feynmankac.1}
	\Zfn(\dev;t,x) 
	&= 
	\E_x \Big[ \exp\Big( \int_0^t \dev(s,\bm(t-s)) \, \d s \Big) \, \delta_0(\bm(t)) \Big]
\\
	\label{e.feynmankac.2}
	&= 
	\E_{0\to x} \Big[ \exp\Big( \int_0^t \dev(s,\bb(s)) \, \d s \Big)  \Big] \hk(t,x).
\end{align}
In \eqref{e.feynmankac.1}, the expectation $ \E_x $ is taken with respect to a Brownian motion that starts from $ x $,
and in \eqref{e.feynmankac.2} the $ \E_{0\to x} $ is taken with respect to a Brownian bridge $ \bb(s) $ that starts from $ \bb(0)=0 $ and ends in $ \bb(t)=x $.
Indeed, the expression \eqref{e.feynmankac.1} is equivalent to \eqref{e.Zfn.expansion}
upon Taylor-expanding the exponential in \eqref{e.feynmankac.1} and exchanging the sum with the expectation.
The exchange is justified by the bound in Lemma~\ref{l.chaosfn.bd.nw}.
Set 
\begin{align}
	\label{e.feynmankac}
	\hfn(\dev;t,x) := \log(\sqrt{4\pi}\Zfn(\dev;t,x)) 
	=
	\log(\sqrt{4\pi}\hk(t,x))
	+	
	\log \E_{0\to x} \Big[ \exp\Big( \int_0^t \dev(s,\bb(s)) \, \d s \Big)  \Big].
\end{align}
Take log on both sides of \eqref{e.feynmankac.1} and insert the result into \eqref{e.rateone-}.
We have 
\begin{align}
	\label{e.onept.rate}
	\rateone(- \scl)
	=
	\inf\big\{ \tfrac12 \normL{ \dev }^2  : \, \hfn(\dev;2,0) \leq -\scl \big\}.
\end{align}

We expect the right hand side of \eqref{e.onept.rate} to grow as $ \scl^{5/2} $ when $ \scl\to\infty $.
As pointed out in \cite{kolokolov07,kolokolov09,meerson16,kamenev16}, such a power law follows from scaling.
More precisely, 
when $ \scl \to \infty $, it is natural to scale $ \hfn \mapsto \scl^{-1} \hfn $ and $ \dev\mapsto \scl \dev $.
Accordingly, for the Brownian bridge in \eqref{e.feynmankac} to complete on the same footing,
it is desirable to have a factor $ \scl^{-1/2} $ multiplying $ \bb(s) $.
This is so because large deviations of $ \scl^{-1/2} \bb(s) $ occurs at rate $ \scl $, which is compatible with the scaling $ \dev\mapsto \scl \dev $.
To implement these scaling, in \eqref{e.feynmankac} replace $ \dev(t,x) \mapsto \scl \dev(t,\scl^{-1/2} x) $ and $ x\mapsto \scl^{1/2} x $
and divide the result by $ \scl $. 
Let $ \hfn_\scl(\dev;t,x) := \scl^{-1} \hfn(\scl\dev(\Cdot,\scl^{-1/2} \Cdot);t,\lambda^{1/2} x) $ denote the resulting function on the left hand side.
We have
\begin{align}
	\label{e.feynmankac.scl}
	\hfn_\scl(\dev;t,x) 
	&= 
	\scl^{-1}\log(\sqrt{4\pi}\hk(t,\scl^{\frac12} x)) + \scl^{-1} \log  \E_{0\to\scl^{1/2} x} \Big[ \exp\Big( \int_0^t \scl\dev(s,\scl^{-\frac12} \bb(s)) \, \d s \Big) \Big].
\end{align}
The replacement $ \dev(t,x) \mapsto \scl \dev(t,\scl^{-1/2} x) $ changes $ \normL{\dev}^2 $ by a factor of $ \scl^{5/2} $,
so \eqref{e.onept.rate} translates into
\begin{align}
\label{e.onept.rate.scl}
	\rateone(-\scl)
	=
	\scl^\frac52
	\inf\big\{ \tfrac12 \normL{ \dev }^2  : \, \hfn_\scl(\dev;2,0) \leq -1 \big\}.
\end{align}

Proving Theorem~\ref{t.main}~\ref{t.main.lower} hence amounts to proving
\begin{align}
\label{e.deep.lower}
	\lim_{\scl\to\infty}
	\big( \inf\big\{ \tfrac12 \normL{ \dev }^2  : \, \hfn_\scl(\dev;2,0) \leq -1 \big\} \big)
	=
	\frac{4}{15\pi}.	
\end{align}

\subsubsection{The optimal deviation $ \devm $ and its geodesics}
\label{s.devm}

We begin by introducing a function $ \devm \in \Lsp^2([0, 2] \times \R) $.
The definition of this function is motivated by physics argument \cite{kolokolov09,meerson16,kamenev16}; see Section~\ref{s.intro.phys}.
In the context of Proposition~\ref{t.FW},
$ \dev $ describes possible deviations of the spacetime white noise $ \sqrt\e\xi $.
Such $ \devm $ is a candidate for the optimal $ \dev $,
so we refer to $ \devm $ as the \textbf{optimal deviation}.

To define $ \devm $, consider the unique $ \Csp^1[1,2) $-valued solution $ r(t) $ of the equation   
\begin{align}
	\label{e.r.ode}
	r'(t) = 2^\frac12 \pi^{-\frac12} \, r^2 \sqrt{r-\pi/2}, & \text{ for } t \in (1,2), \quad r(1)=\pi/2, \quad \text{and } r|_{(1,2)}>\pi/2,
\end{align}
and symmetrically extend it to $ \Csp^1(0,2) $ by setting $ r(t) := r(2-t) $ for $ t\in(0,1) $.
Integrating \eqref{e.r.ode} gives
\begin{align}
	\label{e.r}
	\frac{(r(t)-\pi/2)^\frac12}{r(t)\pi/2} + (\tfrac{2}{\pi})^{\frac32} \arctan\Big( \big( \tfrac{r(t)}{\pi/2}-1 \big)^{\frac12} \Big) = (\tfrac{2}{\pi})^{\frac12} \, |t-1|.
\end{align}
Let us note a few useful properties of $ r(t) $.
It can be checked from \eqref{e.r} that $ \lim_{s\downarrow 0}r(s)=\lim_{s\uparrow 2}r(s)=+\infty $.
The integral $ \int_0^2 r(t) \, \d t= 2 \int_1^2 r(t) \, \d t $ can be evaluated with the aid of \eqref{e.r.ode}:
perform the change of variables $ 2 \int_1^2 r(t) \, \d t = 2\int_{\pi/2}^\infty \frac{r}{r'(t)} \d r $
and use \eqref{e.r.ode} to substitute $ r'(t) $. The result reads
\begin{align}
	\label{e.r.int}
	\int_0^2 r(t) \, \d t
	=
	\int_0^2 |r(t)| \, \d t
	=
	2\pi.	
\end{align}
Set $ \ell(t): = 1/r(t) $ for $ t\in(0,2) $, and let $ \ell(0):=0 $ and $ \ell(2):=0 $ so that $ \ell\in\Csp[0,2] $. We define
\begin{align}
	\label{e.devm}
	\devm(t,x) := - \frac{r(t)}{2\pi} \,\Big( 1 - \frac{x^2}{\ell(t)^2} \Big)_+.
\end{align}

Next, setting $ \dev=\devm $ in \eqref{e.feynmankac}, we seek to characterize the $ \scl\to\infty $ limit of the resulting function:
\begin{align}
\label{e.hm}
	\hm(t,x) := \lim_{\scl\to\infty} \hfn_\scl(\devm;t,x),
\end{align}
for all $ (t,x)\in(0,2]\times\R $.
Even though only $ \hm(2,0) $ will be relevant toward the proof of \eqref{e.deep.lower},
we treat general $ (t,x)\in(0,2]\times\R $ for its independent interest.
\begin{rmk}
Indeed, with $ \devm $ being the optimal deviation of the spacetime white noise,
the function $ \hm $ should be viewed as the limit shape of 
$ \hh_{\e,\scl}(t,x) := \scl^{-1}\log \ZZe(t,\scl^{1/2}x) $ under the conditioning $ \{ \hh_{\e,\scl}(0,2) \leq -1 \} $ with $ \scl \gg 1 $. A explicit expression of $ \hm(1,x) $ is given in \cite{hartmann2019optimal}.
One can show that \cite[Eq's~(10)-(11)]{hartmann2019optimal} 
coincide with the variational expression of $ \hm $ given in \eqref{e.minimizing.path} below.

\emph{Proving} that $ \hm $ is the limit shape of $ h_{\e,\scl} $ remains open, which we leave for future work.

\end{rmk}

To characterize \eqref{e.hm}, we first turn the limit into certain minimization problem over paths, by using Varadhan's lemma.
To setup notation, we let $ \Hsp^1_{0,x}[0,t] $ denote the space of $ \Hsp^1 $ functions on $ [0,t] $ such that $ \Path(0)=0 $ and $ \Path(t)=x $,
and likewise for $ \Csp_{0,x}[0,t] $.
For $ \Path\in\Hsp^1_{0,x}[0,t] $, set
\begin{align}
	\label{e.energy}
	\energy(\Path;t,x) = \int_0^t \tfrac{1}{2} \Path'(s)^2 - \devm(s,\Path(s)) \ \d s.
\end{align}
\begin{lem}
\label{l.devm.varadhan}
For any $ (t,x)\in(0,2]\times\R $,
\begin{align}
	\label{e.l.devm.varadhan}
	\lim_{\scl\to\infty} \hfn_\scl(\devm;t,x) =: \hm(t,x) = - \inf\big\{ \energy(\Path;t,x) : \Path\in\Hsp^1_{0,x}[0,t] \big\}.
\end{align}
\end{lem}
\begin{proof}
Let $ F(\Path) := \int_0^t  \devm(s,\Path(s))\, \d s $.
In \eqref{e.feynmankac.scl}, set $ \dev\mapsto\devm $ and let $ \scl\to\infty $ to get
\begin{align}
\label{e.varadhan.}
\lim_{\scl\to\infty} \hfn_\scl(\devm;t,x)
=
-\tfrac{x^2}{2t}
+
\lim_{\scl\to\infty}
\scl^{-1} \log \E_{0\to \scl^{1/2} x} \big[ \exp\big( \scl F(\scl^{-\frac12} \bb(s)) \big) \big].
\end{align}
We have assumed that the last limit exists.
To prove the existence of the limit and to evaluate it we appeal to Varadhan's lemma.
To start, let us establish the \ac{LDP} for $ \{ \scl^{-1/2} \bb(s) : s\in[0,t] \} $. 
Express $ \bb $ as $ \bb(s) = \bm(s) + (x-\bm(t)) s /t $, where $ \bm $ denotes a standard Brownian motion.
Since the map $ \Path \mapsto \Path +(x-\Path(t))s/t $ from $ \{ \Path\in\Csp[0,t] : \Path(0)=0\} $ to $ \Csp_{0,x}[0,t] $  is continuous,
we can use the contraction principle to push forward the \ac{LDP} for $ \scl^{-1/2}\bm $.
The result asserts that $ \scl^{-1/2}\bb $ enjoys an \ac{LDP} with speed $ \scl $ and the rate function
$	
I_\textbb(\Path) := \inf\{ \frac12 \int_0^t (\Path'(s)-v-\frac{x}{t})^2 \d s : v\in \R \}
$ 
for $ \Path \in \Hsp^1_{0,x}[0,t] $
and $ I_\textbb(\Path) = +\infty $ otherwise.
Optimizing over $ v\in\R $ gives
\begin{align*}
I_\textbb(\Path) 
= 
\left\{\begin{array}{l@{,}l}
\int_0^t \frac12\Path'(s)^2 \d s - \frac{x^2}{2t} &\text{ for } \Path \in \Hsp^1_{0,x}[0,t],
\\
+\infty	&\text{ for } \Path \in \Csp_{0,x}[0,t]\setminus \Hsp^1_{0,x}[0,t].
\end{array}\right.
\end{align*}

To apply Varadhan's lemma we need to check, for $ F(\Path) := \int_0^t  \devm(s,\Path(s))\, \d s $:
\begin{enumerate}[leftmargin=16pt, label=(\roman*),itemsep=5pt]
\item \label{l.devm.varadhan.conti} $ F: \Csp_{0,x}[0,t] \to \R $ is continuous.\\
This statement would follow if $ \devm $ were uniformly continuous on $ [0,t]\times\R $.
The function  $ \devm(s,y) $ however is discontinuous at $ (0,0) $ and $ (2,0) $.
To circumvent this issue, for small $ \delta>0 $, we consider the truncation $ \devm^\delta(s,y) := \ind_{\set{|s-1|<1-\delta}}\devm(s,y) $.
The truncated functional $ F_\delta(\Path):=\int  \devm^\delta(t,\Path(t))\, \d t  $ is continuous on $ \Csp_{0,x}[0,t] $.
The difference $ F -F_\delta $ is bounded by $ |(F-F_\delta)(\Path)| \leq \int_{|s-1|>1-\delta} |\devm(s,\Path(s))|\, \d s \leq \frac{1}{2\pi} \int_{|s-1|>1-\delta} |r(s)| \d s $.
By \eqref{e.r.int}, the last expression converges to zero as $ \delta\to 0 $, \emph{uniformly} in $ \Path\in\Csp_{0,x}[0,t] $.
From these properties we conclude that $ F: \Csp_{0,x}[0,t] \to \R $ is continuous.
\item $ \displaystyle \lim_{M\to\infty} \limsup_{\scl\to\infty} \scl^{-1} \log \E_{0\to x }\big[ \exp\big(\scl F(\scl^{-1/2}\bb)\big) \ind\set{F(\scl^{-1/2}\bb)>M}\big] = -\infty  $\\
This holds since $ \devm \leq 0 $, which implies $ F \leq 0 $.
\end{enumerate}

Varadhan's lemma applied to the last term in \eqref{e.varadhan.} completes the proof.
\end{proof}

Lemma~\ref{l.devm.varadhan} expresses $ \hm(t,x) $ in terms of a variational problem over paths.
We refer to the minimizing path(s) in \eqref{e.l.devm.varadhan} (if exists) as a \textbf{geodesic}.
The next step is to identify the geodesic.
Let 
\begin{align*}
	\Omega := \{ (s,y) : s\in[0,2], |y| \leq \ell(s) \} 
\end{align*}
denote the support of $ \devm $, with the boundary $ \partial\Omega = \{ (s,y) : t\in[0,2], |y| =\ell(s) \} $.

\begin{prop}
\label{p.geodesic}
\begin{enumerate}[leftmargin=20pt, label=(\alph*)]
\item[]
\item\label{p.geodesic.exists} For any $ (t,x)\in(0,2]\times\R $, the infimum
\begin{align}
\label{e.minimizing.path}
\hm(t,x) = 
-\inf\big\{ \energy(\Path;t,x) : \Path\in\Hsp^1_{0,x}[0,t] \big\}
\end{align}
is attended in $ \Hsp^1_{0,x}[0,t] $.
\item\label{p.geodesic.00} When $ (t,x)=(2,0) $, the geodesics are $ \alpha \ell(\Cdot) $, $ |\alpha| \leq 1 $.
\item\label{p.geodesic.in} When $ (t,x) \in \Omega\cap\{t\in(0,2)\} $, the unique geodesic is $ (x/\ell(t)) \ell(\Cdot) $.
\item\label{p.geodesic.out} When $ (t,x) \in \Omega^\mathrm{c}\cap\{t\in(0,2]\} $, is the geodesic is the unique $ \Csp^1_{0,x}[0,t] $ path such that $ \Path|_{[0,t_*]}=\ell|_{[0,t_*]} $ and $ \Path|_{[t_*,t]} $ is linear, for some $ t_*\in(0,t) $.
\end{enumerate}
\end{prop}
See Figure~\ref{f.geodesics} for an illustration for these geodesics.

\begin{figure}[h]
\includegraphics[width=.5\linewidth]{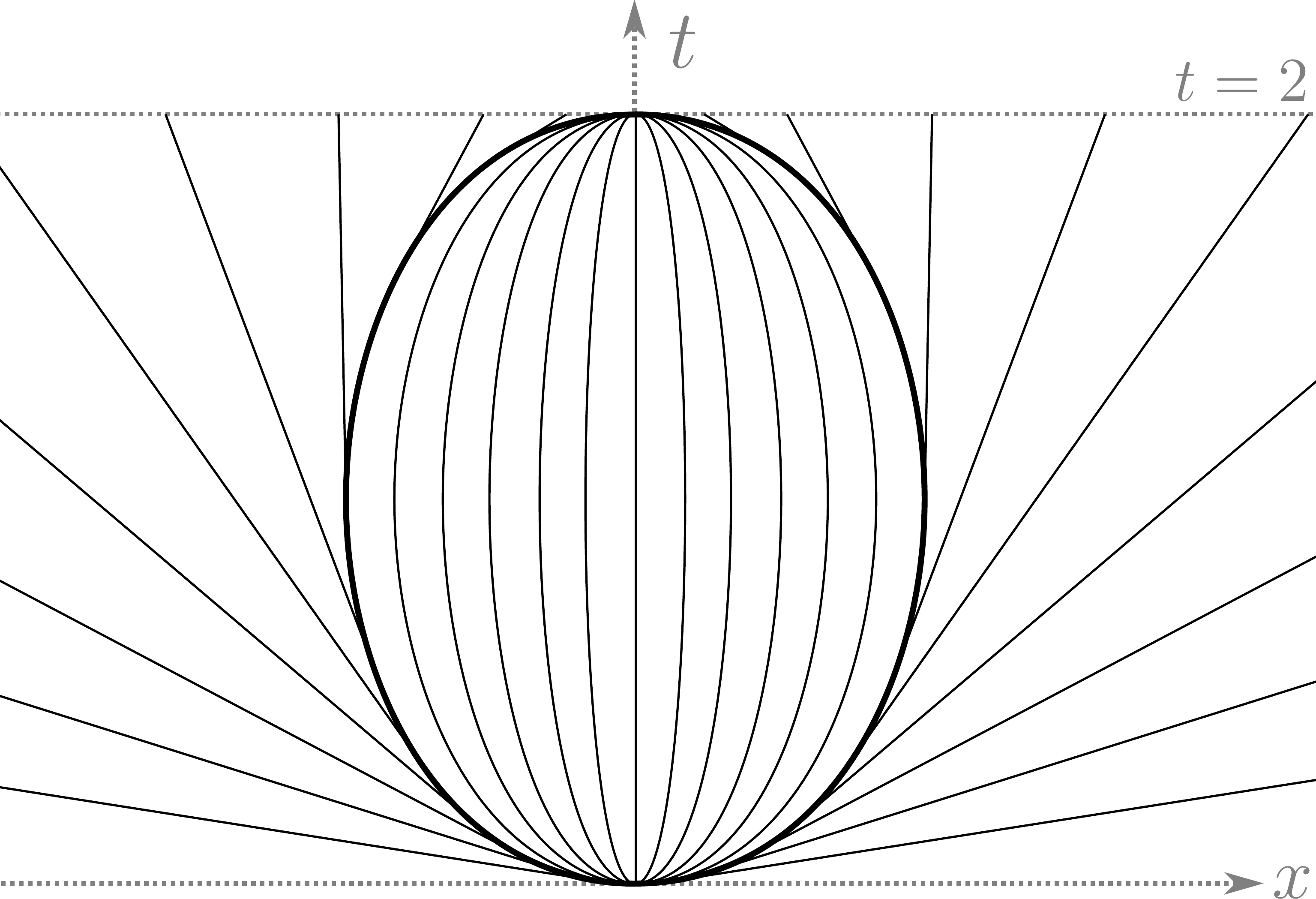}
\caption{%
The solid curves are the geodesics for \protect\eqref{e.minimizing.path},
with the thick ones being $ \pm\ell(\protect\Cdot) $.
Those geodesics outside $ \pm\ell(\protect\Cdot) $ are linear, and touch $ \pm\ell(\protect\Cdot) $ at tangent.%
}
\label{f.geodesics}
\end{figure}

\begin{rmk}
\label{r.geodesic}
An intriguing feature of Proposition~\ref{p.geodesic}\ref{p.geodesic.00} is the \emph{nonuniqueness} of the geodesics between $ (0,0) $ and $ (2,0) $.
For any $ |\alpha| \leq 1 $, $ \Path=\alpha\ell $ is one such geodesic, so the paths span a lens-shaped region $ \Omega $.
For the exponential \ac{LPP}, \cite{basu19} proved that the point-to-point geodesic (in the context of \ac{LPP}) does not concentrate around any given path under a lower-tail conditioning.
Though the setups differ, the result of \cite{basu19} and Proposition~\ref{p.geodesic}\ref{p.geodesic.00} are consistent.
It is an intriguing question to explore deeper connection between these two phenomena.
For example, is it true that for \ac{LPP} under lower-tail conditioning, the distribution of the geodesic spans a lens-like region?
\end{rmk}

To streamline the proof of Proposition~\ref{p.geodesic}, let us prepare a few technical tools.
The Euler--Lagrangian equation for \eqref{e.energy} is
\begin{align}
\label{e.euler.lang}
\Path'' = - \partial_x \devm(s,\Path(s))
=
\left\{\begin{array}{l@{,}l}
- \tfrac{r(s) }{\pi \ell(s)^2} \Path & \text{ when } (s,\Path(s)) \in \Omega^\circ,
\\
0	& \text { when } (s,\Path(s)) \in \Omega^\mathrm{c}.
\end{array}\right.
\end{align}
The equation~\eqref{e.euler.lang} is ambiguous when $ (s,\Path(s)) \in \partial\Omega $ because $ \partial_x\devm $ is not continuous there. 
We will avoid referencing \eqref{e.euler.lang} when $ (s,\Path(s)) \in \partial\Omega $.
It will be convenient to also consider 
\begin{align}
\label{e.euler.lang.}
\Path'' = - \tfrac{r(s) }{\pi \ell(s)^2} \Path,
\end{align}
which coincides with \eqref{e.euler.lang} in $ \Omega^\circ $.
\begin{lem}
\label{l.geodesic}
\begin{enumerate}[leftmargin=20pt, label=(\alph*)]
\item[]
\item\label{l.geodesic.concave} The function $ \ell $ is strictly concave and $ \lim_{s\downarrow 0} |\ell'(s)| = +\infty $.
\item\label{l.geodesic.euler.lang.} For any $ \alpha\in\R $, the function $ \alpha\ell(s) $ solves \eqref{e.euler.lang.} for $ s\in(0,2) $.
\item\label{l.geodesic.energy} For any for any $ |\alpha|\leq 1 $, $ \energy(\alpha\ell;2,0) = -1 $. 
\item\label{l.geodesic.euler.lang} In $ (\partial\Omega)^\mathrm{c} $, any geodesic of \eqref{e.minimizing.path} is $ \Csp^2 $ and solves \eqref{e.euler.lang}.
\item\label{l.geodesic.within} When $ (t,x)\in\Omega $, any geodesic of \eqref{e.minimizing.path} lies entirely in $ \Omega $.
\item\label{l.geodesic.C1} Let $ \Path \in \Hsp^1_{0,x}[0,t] $ be a geodesic of \eqref{e.minimizing.path},
and consider $ (t_*,\Path(t_*)) \in \partial \Omega $ with $ t_*\in(0,t) $. Then
\begin{align*}
\lim\limits_{\beta\downarrow 0 } \Big( \frac{1}{\beta} \int_{t_*}^{t_*+\beta} \Path'(s)\d s - \frac{1}{\beta} \int_{t_*-\beta}^{t_*} \Path'(s)\d s \Big)
=
0.
\end{align*}	
\end{enumerate}
\end{lem}

\begin{proof}
Parts~\ref{l.geodesic.concave}--\ref{l.geodesic.energy} follow by straightforward calculations from $ \ell(s)=1/r(s) $, \eqref{e.r.ode}, and \eqref{e.r.int}.
Part~\ref{l.geodesic.euler.lang} follows by standard variation procedure.

\ref{l.geodesic.within}
The geodesic $ \Path $ starts and ends within $ \Omega $, i.e., $ (0,\Path(0)) =(0,0)\in\Omega $ and $ (t,\Path(t)) = (t,x) \in \Omega $.
If the geodesic ever leaves $ \Omega $,
then there exists $ t_1<t_2\in[0,t] $ such that $ \Path|_{(t_1,t_2)} $ lies outside $ \Omega $ 
and $ (t_i,\Path(t_i)) \in \partial \Omega $ for $ i=1,2 $.
See Figure~\ref{f.geo.outside} for an illustration.
Let us compare the functional $ \energy(\Cdot;t,x) $ (c.f., \eqref{e.energy}) restricted onto the segments $ \Path|_{[t_1,t_2]} $ and $ \pm\ell|_{[t_1,t_2]} $,
where the $ \pm $ sign depends on which side of the boundary $ (t_1,\Path(t_1)) $ and $ (t_2,\Path(t_2)) $ belong to, c.f., Figure~\ref{f.geo.outside}.
First $ \devm $ vanishes along both segments. 
Next, the strict concavity of $ \ell $ from Part~\ref{l.geodesic.concave} implies $ \int_{t_1}^{t_2} \Path'(s)^2 \d s > \int_{t_1}^{t_2} \ell'(s)^2 \d s $.
Therefore, we can modify $ \Path $ by replacing the segment $ \Path|_{[t_1,t_2]} $ with $ \pm\ell|_{[t_1,t_2]} $ to decreases the value of $ \energy(\Path;2,0) $.
This contradicts with assumption that $ \Path $ is a geodesic. 
Hence the geodesic must stay completely within $ \Omega $.

\begin{figure}
\centering
\begin{minipage}{0.35\textwidth}
	\includegraphics[width=\textwidth]{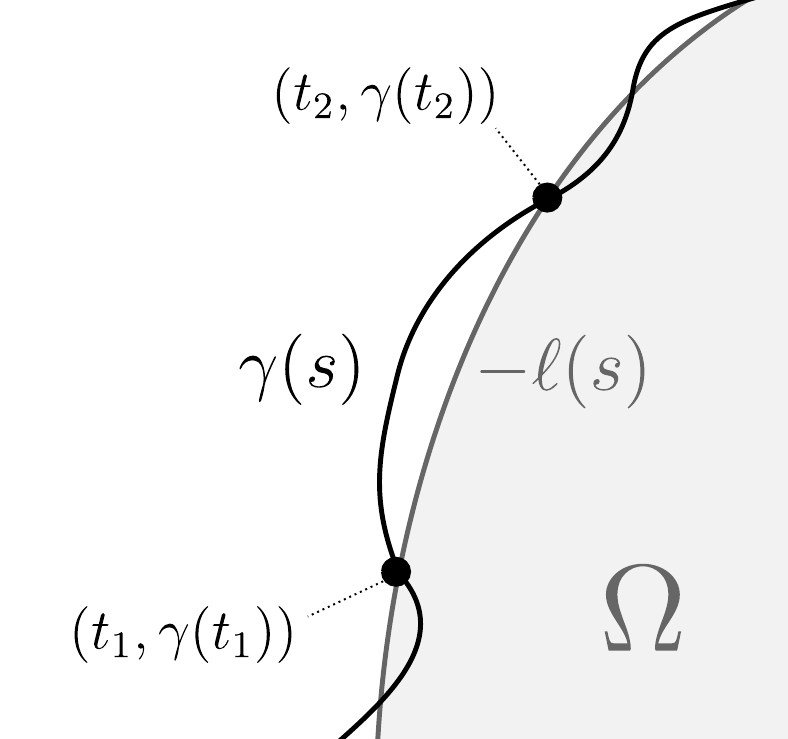}
\end{minipage}
\hspace{20pt}
\begin{minipage}{0.35\textwidth}
	\includegraphics[width=\textwidth]{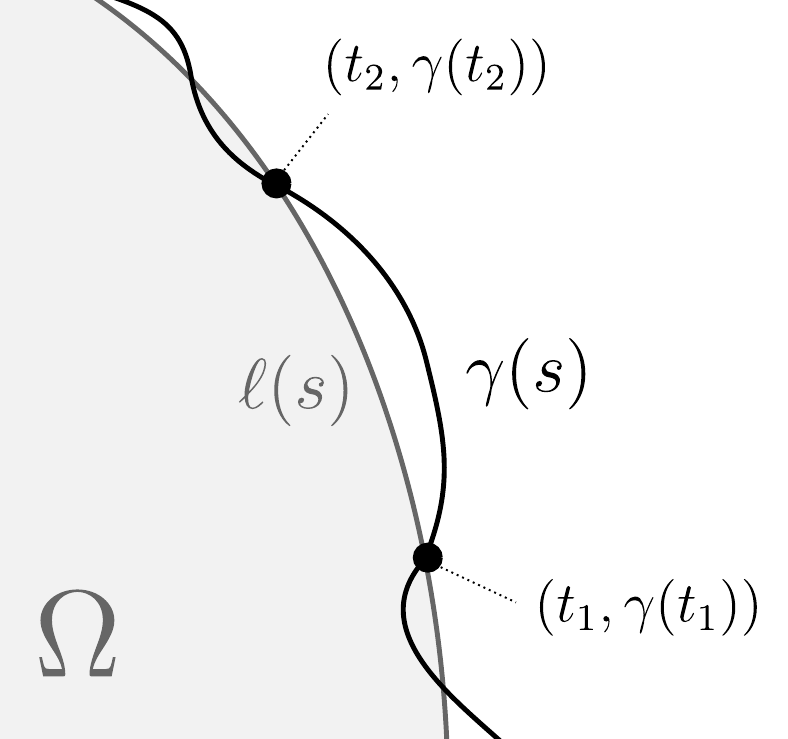}
\end{minipage}
\caption{Illustration of Part~\protect\ref{l.geodesic.within} of the proof of Lemma~\protect\ref{l.geodesic}}
\label{f.geo.outside}
\end{figure}


\ref{l.geodesic.C1} The idea is to perform variation.
Fix a neighborhood $ O $ of $ t_* $ with $ \bar{O}\subset(0,2) $. For $ f\in\Csp^\infty_c(O) $ consider
\begin{align*}
F(\alpha):=
\int_0^t \tfrac{1}{2}(\Path'+\alpha f')^2 - \devm(s,\Path+\alpha f) \ \d s.
\end{align*} 
The derivative $ \partial_x\devm $ is bounded on $ \bar{O}\times\R $ (even though not continuous).
Taylor expanding $ F $ around $ \alpha=0 $ then gives
$ \int \Path'(s) f'(s) \d s \leq c \int |f(s)| \d s $, for some constant $ c<\infty $.
Within the last inequality,
substitute $ f(s)\mapsto f(s+u) $, integrate the result over $ u\in[-\frac12\beta,\frac12\beta] $, and divide both sides by $ \beta $.
This gives
\begin{align*}
\frac{1}{\beta} \int \Path'(s) (f(s+\tfrac12\beta)-f(s-\tfrac12\beta)) \d s 
=
\frac{1}{\beta} \int (\Path'(s-\tfrac12 \beta) -\Path'(s+\tfrac12 \beta)) f(s) \d s 
\leq
c \int |f(s)|\d s.
\end{align*}
This inequality holds for smooth $ f(s) $ supported in $ \{s:s\pm\frac12\beta \in O\} $.
Since $ \Path'\in\Lsp^2[0,t] $, the equality extends to $ f\in\Lsp^2 $.
Specializing $ f= \pm\ind_{(t_*-\frac12\beta,t_*+\frac12\beta)}$
and taking $ \beta \downarrow 0 $ gives the desired result.
\end{proof}

\begin{proof}[Proof of Proposition~\ref{p.geodesic}] 
\ref{p.geodesic.exists}
The proof follows from standard argument of the direct method.
Take any minimizing sequence $ \{ \Path_n \} $.
For such a sequence, $ \{\Path'_n\} $ is bounded in $ \Lsp^2[0,t] $.
By the Banach--Alaoglu theorem, after passing to a subsequence we have $ \Path'_n \to \cuttime \in \Lsp^2[0,t] $ weakly in $ \Lsp^2[0,t] $.
Let $ \Path(\bar{s}) := \int_0^{\bar{s}} \cuttime(s) \d s $. 
We then have $ \Path_n \to \Path $ in $ \Csp_{0,x}[0,t] $ and $ \int_0^t \Path'(s)^2 \d s = \normL{\cuttime}^2 \leq \lim_{n}\normL{\Path'_n}^2 $.
Also, by Property~\ref{l.devm.varadhan.conti} in the proof of Lemma~\ref{l.devm.varadhan},
$ \int_0^t \devm(s,\Path_n(s))\d s \to \int_0^t \devm(s,\Path(s))\d s  $.
We have verified that $ \Path \in \Hsp^1_{0,x}[0,t] $ a geodesic.

\ref{p.geodesic.00}
The proof amounts to showing that any geodesic must be of the form $ \alpha\ell $, for some $ |\alpha|\leq 1 $.
Once this is done, Lemma~\ref{l.geodesic}\ref{l.geodesic.energy} guarantees that any such path is a geodesic.

We begin with a reduction.
For a geodesic $ \Path\in\Hsp^1_{0,0}[0,2] $,
consider its first and second halves $ \Path_1 := \Path|_{[0,1]} $ and $ \Path_2(s) := \Path(2-s)|_{s\in[0,1]} $.
Joining each half with itself end-to-end gives the symmetric paths $ \bar{\Path}_i(s) := \Path_i(s)\ind_{[0,1]}(s) + \Path_i(s-1)\ind_{(1,2]}(s) $, for $ s\in[0,2] $ and $ i=1,2 $.
These symmetrized paths are also geodesics.
To see why, note that since $ \devm(s,y) $ is symmetric around $ s=1 $, we have $ \energy(\bar{\Path}_i;2,0) = 2 \energy(\Path_{i};1,\Path(1)) $, for $ i=1,2 $,
and $ \energy(\Path;2,0) = \energy(\Path_{1};1,\Path(1)) + \energy(\Path_{2};1,\Path(1)) $.
On the other hand, $ \Path $ being a geodesic implies $ \energy(\Path;2,0) \leq \energy(\bar{\Path}_i;2,0) $, for $ i=1,2 $.
From the these relations we infer that $ \energy(\bar{\Path}_1;2,0)=\energy(\bar{\Path}_2;2,0) = \energy(\Path;2,0) $,
namely, the symmetrized paths $ \bar{\Path}_1 $ and $ \bar{\Path}_2 $ are also geodesics.
Recall that our goal is to show any geodesic must be of the form $ \alpha\ell $, for some $ |\alpha|\leq 1 $.
If we can establish the statement for $ \bar{\Path}_1 $ and $ \bar{\Path}_2 $, the same immediately follows for $ \Path $. 
Hence, without loss of generality, hereafter we consider only symmetric geodesics.

Fix a geodesic $ \Path\in\Hsp^1_{0,0}[0,2] $.
As argued in the preceding paragraph, we can and shall assume $ \Path(s) $ is symmetric around $ s=1 $,
and by Lemma~\ref{l.geodesic}\ref{l.geodesic.within} the path lies entirely in $ \Omega $.
The last condition implies $ |\Path(1)| \leq \ell(1) $.
Consider first the case $ |\Path(1)| < \ell(1) $. 
By Lemma~\ref{l.geodesic}\ref{l.geodesic.euler.lang},
within a neighborhood of $ s=1 $ the path $ \Path(s) $ is $ \Csp^2 $ and solves \eqref{e.euler.lang} and therefore \eqref{e.euler.lang.}.
The symmetry of $ \Path $ gives $ \Path'(1)=0 $.
The uniqueness of the ODE \eqref{e.euler.lang.} and Lemma~\ref{l.geodesic}\ref{l.geodesic.euler.lang.} now imply $ \Path(s) = \alpha \ell(s) $,
for $ \alpha = \Path(1)/\ell(1) $ and for all $ s $ in a neighborhood of $ s=1 $.
This matching $ \Path(s) = \alpha \ell(s) $ extends to $ s\in(0,2) $ by standard continuity argument.
This concludes the desired result for the case $ |\Path(1)| < \ell(1) $.

Turning to the case $ |\Path(1)|=\ell(1) $, we need to show $ \Path=\pm\ell $. 
Let us argue by contradiction.
Assuming the contrary, we can find $ t_2\in(0,1)\cup(1,2) $ such that $ (t_2,\Path(t_2)) \in \Omega^\circ $. By the symmetry of $ \Path $ around $ s=1 $ we can and shall assume $ t_2\in(1,2) $.
Tracking along $ \Path $ backward in time from $ t_2 $, we let $ t_* := \inf\{ s\in[0,t_*] : |\Path(s)| < \ell(s) \} $ be the first hitting time of $ \partial\Omega $.
Indeed $ t_* \in [1,t_2) $ and $ \Path(t_*) = \pm \ell(t_*) $.
Let us take `$ + $' for simplicity of notation; 
see Figure~\ref{f.geo.touching} for an illustration.
The case for `$ - $' can be treated by the same argument.
By Lemma~\ref{l.geodesic}\ref{l.geodesic.euler.lang}, $ \Path|_{(t_*,t_2)} $ solves \eqref{e.euler.lang} and therefore \eqref{e.euler.lang.}.
On the other hand, $ \ell $ also solves \eqref{e.euler.lang.} by Lemma~\ref{l.geodesic}\ref{l.geodesic.euler.lang.}.
These facts along with the well-posedness of \eqref{e.euler.lang.} at $ (t_*,\ell(t_*)) $ imply that $ \Path|_{[t_*,t_2)} \in \Csp^2[t_*,t_2) $
and $ \lim_{\beta\downarrow 0}\Path'(t_*+\beta) \neq \ell'(t_*) $.
Either `$ < $' or `$ > $' holds between these two quantities.
The property $ \{(t,\Path(t))\}_{t\in(t_*,t_2)} \subset \Omega^\circ $ tells us that it is `$ < $', namely $ \lim_{\beta\downarrow 0}\Path'(t_*+\beta) < \ell'(t_*) $.
Combining this inequality with Lemma~\ref{l.geodesic}\ref{l.geodesic.C1} gives
$ \lim_{\beta\downarrow 0} \frac{1}{\beta}\int_{t_*-\beta}^{t_*} \Path'(s) \d s = \lim_{\beta\downarrow 0} \frac{1}{\beta}(\ell(t_*)-\Path(t_*-\beta)) < \ell'(t_*) $.
Recall from Lemma~\ref{l.geodesic}\ref{l.geodesic.concave} that $ \ell $ is concave.
The last inequality then forces $ \Path(t_*-\beta) > \ell(t_*-\beta) $ for all small enough $ \beta>0 $.
This statement contradicts with the fact that $ \Path $ lies within $ \Omega $.
We have reached a contradiction and hence completed the proof for the case $ |\Path(1)|=\ell(1) $. 

\begin{figure}[h]
\centering
\includegraphics[width=.25\textwidth]{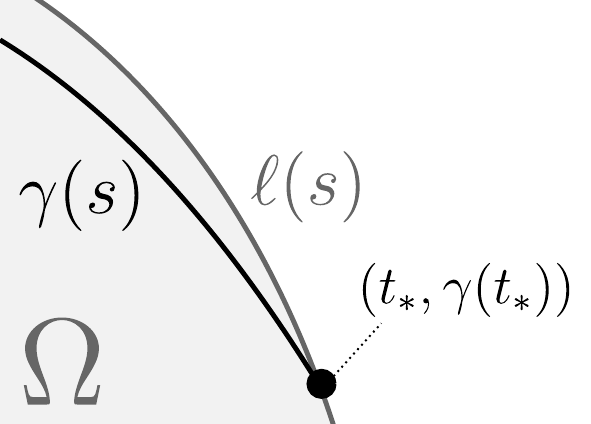}
\caption{Illustration of Part~\protect\ref{p.geodesic.00} of the proof of Proposition~\protect\ref{p.geodesic}.
Only the portion $ s\geq t_* $ of the curve $ \Path(s) $ is shown.%
}
\label{f.geo.touching}
\end{figure}

\ref{p.geodesic.in}
Our goal is to characterize the geodesic between $ (0,0) $ and $ (t,x) $.
The idea is to `embed' such a minimization problem into a minimization problem between $ (0,0) $ and $ (2,0) $.
More precisely consider
\begin{align}
\label{e.mini.tx}
\inf\big\{ \energy(\Path;2,0) : \Path \in \Hsp^1_{0,x}[0,2], \ \Path(t)=x  \big\}.
\end{align}
The infimum is taken over all $ \Hsp^1 $ path that joins $ (0,0) $ and $ (2,0) $ \emph{and} passes through $ (t,x) $.
Such an infimum can be divided into two parts as
\begin{align}
\label{e.mini.tx.}
\eqref{e.mini.tx}
=
\inf\big\{ \energy(\Path;t,x) : \Path \in \Hsp^1_{0,x}[0,t]  \big\}
+
\inf\Big\{ \int_t^2 \tfrac{1}{2} \Path'(s)^2 - \devm(s,\Path(s)) \ \d s : \Path \in \Hsp^1_{x,0}[t,2]  \Big\}.
\end{align}
Take any geodesic $ \Path \in \Hsp^1_{0,x}[0,t]  $ for the first infimum in \eqref{e.mini.tx.} and any geodesic $ \bar{\Path} \in \Hsp^1_{x,0}[t,2] $ for the second infimum in \eqref{e.mini.tx.}.
(The existence of such geodesics can be established by the same argument in Part~\ref{p.geodesic.exists}.)
The concatenated path $ \Path_\text{c}(s) := \Path(s)\ind_{s\in[0,t]} + \bar{\Path}(s)\ind_{s\in(t,2]} $ is a geodesic for \eqref{e.mini.tx}.
Hence $ \energy(\Path_\text{c};2,0) \geq \energy(\til\Path;2,0) $, for any $ \til\Path \in \Hsp^1_{0,0}[0,2] $ that passes through $ (t,x) $.
Set $ \alpha=x/\ell(t) $. The last inequality holds in particular for $ \til\Path= \alpha\ell $.
On the other hand, under current assumption $ (t,x)\in\Omega $, we have $ |\alpha| \leq 1 $,
so Part~\ref{p.geodesic.00} asserts that $ \alpha\ell $ minimizes \eqref{e.mini.tx} even \emph{without} the constraint $ \Path(t)=x $.
Therefore, $ \energy(\Path_\text{c};2,0) = \energy(\alpha\ell;2,0) $, and $ \Path_\text{c} $ itself is a geodesic for 
$ \inf\{ \energy(\Cdot;0,2) : \til\Path \in \Hsp^1_{0,0}[0,2] \} $.
The last statement and Part~\ref{p.geodesic.00} force $ \Path_\text{c} = \alpha\ell $, which concludes the desired result.

\ref{p.geodesic.out}
Fix a geodesic $ \Path\in\Hsp^1_{0,x}[0,t] $.
By Lemma~\ref{l.geodesic}\ref{l.geodesic.euler.lang} and the fact that $ (\partial_x \devm)|_{\Omega^\mathrm{c}} = 0 $, the path $ \Path $ is linear outside $ \Omega $.
Tracking along $ \Path $ backward in time from $ t $,
we let $ t_* := \inf\{ s\in[0,t] : |\Path(s)| > \ell(s) \} > 0 $ be the first hitting time of the boundary.
By Lemma~\ref{l.geodesic}\ref{l.geodesic.concave} must have $ t_* > 0 $.
The segment $ \Path|_{[0,t_*]} $ is itself is a geodesic for $ \energy(\Cdot;t_*,\Path(t_*)) $.
Since $ (t_*,\Path(t_*))=(t_*,\pm\ell(t_*)) \in \Omega $,
Part~\ref{p.geodesic.in} implies that $ \Path|_{[0,t_*]} = \pm \ell|_{[0,t_*]} $.
The path $ \Path $ is $ \Csp^1 $ except possibly at $ s=t_* $, 
but Lemma~\ref{l.geodesic}\ref{l.geodesic.C1} guarantees that $ \Path(s) $ is also $ \Csp^1 $ at $ s=t_* $.
For the given $ (t,x)\in\Omega^\mathrm{c} $, there is exactly one $ t_*\in(0,t) $ that satisfies all the prescribed properties, so we have identified the unique geodesic $ \Path $.
\end{proof}

Given Lemma~\ref{l.devm.varadhan} and Proposition~\ref{p.geodesic},
it is possible to evaluate $ \hm(t,x) $ by calculating $ \energy(\Path;t,x) $ along the geodesic(s) given in Proposition~\ref{p.geodesic}.
In particular, Proposition~\ref{p.geodesic}\ref{p.geodesic.00} and Lemma~\ref{l.geodesic}\ref{l.geodesic.energy} gives
\begin{align}
\label{e.hm20}
	\hm(2,0)
	:=
	\lim_{\scl\to\infty} \hfn_\scl(\devm;2,0)
	=
	-1.
\end{align}
Also, straightforward calculations from \eqref{e.devm} (with the help of \eqref{e.r.int}) gives $ \frac{1}{2} \normL{\devm}^2 = \frac{4}{15\pi} $.

We are now ready to prove one side of the inequalities in \eqref{e.deep.lower}, namely
\begin{align}
	\label{e.deep.lower<}
	\limsup_{\scl\to\infty}
	\big( \inf\big\{ \tfrac12 \normL{ \dev }^2  : \, \hfn_\scl(\dev;2,0) \leq -1 \big\} \big)
	\leq
	\tfrac{1}{2} \normL{\devm}^2 
	=
	\tfrac{4}{15\pi}.	
\end{align}
To show \eqref{e.deep.lower<} we would like to have $ \hfn_\scl(\devm;2,0) \leq -1 $ for all large enough $ \scl $,
but \eqref{e.hm20} only gives the inequality for $ \scl=+\infty $.
We circumvent this issue by scaling.
Fix $ \kappa>1 $ and let $ (\devm)_\kappa(t,x) := \kappa\devm(t,\kappa^{1/2}x) $.
Referring to the scaling from \eqref{e.feynmankac} to \eqref{e.feynmankac.scl}, we see that $  \hfn_\scl((\devm)_\kappa;2,0) = \kappa \hfn_\scl(\devm;2,0) $.
This identity together with \eqref{e.hm20} implies $ \hfn_\scl((\devm)_\kappa;2,0) < -1 $ for all large enough $ \scl $.
On the other hand, $ \frac12\normL{(\devm)_\kappa}^2 = \frac{\kappa^{5/2}}2\normL{\devm}^2 $,
so the left hand side of \eqref{e.deep.lower<} is at most $ \frac{\kappa^{5/2}}2\normL{\devm}^2 $.
Letting $ \kappa\downarrow 1 $ concludes \eqref{e.deep.lower<}.

\subsubsection{The reverse inequality}
\label{s.showing>}
To prove \eqref{e.deep.lower}, it now remains only to show the reverse inequality.
Fix any $ \dev \in \Lsp^2([0,2]\times\R) $ with $ \hfn_\scl(\dev;2,0) \leq -1 $.

The first step is to relate $ \hfn_\scl(\dev;2,0) $ to the functional $ \energy(\Path;2,0) $, c.f., \eqref{e.energy}.
Within \eqref{e.feynmankac.scl}, set $ (t,x)\mapsto (2,0) $, 
express the Brownian bridge as $ \bb(t) = \bm(t) - t \bm(2)/2 $, where $ \bb $ denotes a standard Brownian motion,
and apply the Cameron--Martin--Girsanov theorem with $ \scl^{1/2}\Path\in\Hsp^1_{0,0}[0,2] $ being the drift/shift.
The result gives
\begin{align*}
	\hfn_\scl(\dev;2,0)
	=
	- \int_0^2 \tfrac{1}{2} \Path'(t)^2 \d t
	+ \scl^{-1} \log  \E_{0 \to 0} \Big[ \exp\Big( \int_0^2 \Big(\scl\dev(t,\Path+\scl^{-\frac12} \bb) \, \d t + \scl^{\frac12}\Path'(t) \d \bm(t)\Big) \Big) \Big].
\end{align*}
Applying Jensen's inequality to the last term yields, for any $ \Path\in\Hsp^1_{0,0}[0,2] $,
\begin{align}
\label{e.after.jensen}
	-1
	\geq 
	\hfn_\scl(\dev;2,0)
	\geq
	-\scl^{-1} \log \sqrt{4\pi} 
	- 
	\int_0^2 
	\tfrac{1}{2} \Path'(t)^2
	- \E_{0 \to 0} \big[ \dev(t,\Path+\scl^{-\frac12} \bb) \big] \ \d t.
\end{align}
On the right hand side, the first term vanishes as $ \scl\to\infty $,
and the second term resemble the functional $ \energy(\Path;2,0) $.
The difference are that $ \dev $ replaces $ \devm $, 
and there is an additional expectation over $ \scl^{-\frac12} \bb $. 

We next use \eqref{e.after.jensen} to derive a useful inequality.
First, recall from Lemma~\ref{l.geodesic}\ref{l.geodesic.energy} that, for all $ |\alpha| \leq 1 $,
\begin{align}
\label{e.after.jensen.}
	-1
	=
	-\energy(\alpha\ell;2,0)
	=
	- 
	\int_0^2 
	\tfrac{1}{2} (\alpha \ell')^2
	- \devm(t,\alpha\ell) \ \d t.
\end{align}
Substitute $ \Path\mapsto \alpha\ell $ in \eqref{e.after.jensen} and subtract \eqref{e.after.jensen.} from the result.
This gives, for all $ |\alpha| \leq 1 $,
\begin{align*}
	\int_0^2 \big( \devm(t,\alpha\ell) - \E_{0 \to 0} \big[ \dev(t,\alpha\ell+\scl^{-\frac12} \bb) \big]  \big)\, \d t
	\geq
	-\scl^{-1} \log \sqrt{4\pi}.
\end{align*}
Multiply both sides by $ -\frac{1}{2\pi}(1-\alpha^2)_+ $ and integrate the result over $ \alpha\in\R $.
On the left hand side of the result, swap the integrals, multiply the integrand by $ 1=r(t)\ell(t) $, and recognize $ -\frac{r(t)}{2\pi}(1-x^2/\ell(t)^2)_+ = \devm(t,x) $.
We have
\begin{align}
\label{e.deduced}
	\int_0^2\int_\R  \devm(t,\alpha\ell) \Big( \devm(t,\alpha\ell) - \E_{0 \to 0} \big[ \dev(t,\alpha\ell+\scl^{-\frac12} \bb) \big]  \Big)\, \ell(t) \d \alpha \d t 
	\leq
	\scl^{-1} \tfrac{15}{16} \log \sqrt{4\pi}.
\end{align}

To see why \eqref{e.deduced} is useful, let us \emph{pretend} for a moment that $ \scl=+\infty $ in \eqref{e.deduced}.
The discussion in this paragraph is informal, and serves merely as a \emph{motivation} for the rest of the proof.
Informally set $ \scl=+\infty $ in \eqref{e.deduced},
and perform the change of variables $ x=\alpha\ell(t) $ on the left hand side. 
The result gives $ \ip{\devm,\devm-\dev} \leq 0 $ and hence $ \normL{\devm}^2 + \normL{\dev-\devm}^2 \leq  \normL{\dev}^2 $.
The last inequality implies $ \normL{\devm}^2  \leq \normL{\dev}^2 $, which is the desired result.

In light of the preceding discussion, we seek to develop an estimate of $ \ip{\devm,\devm-\dev} $.
To alleviate heavy notation we will often abbreviate $ \scl^{-1/2}\bb =: \textbb $.
Write 
$	
	\ip{\devm,\devm-\dev}
	=
	\int (\devm^2 - \devm \dev)(t,x) \d x \d t.
$
Within the integral add and subtract 
$ \E[\devm^2(t,x-\textbb)] $ and $ \E[\devm(t,x-\textbb)\dev(t,x)] $.
This gives $ \ip{\devm,\devm-\dev} = A_1 + A_2 + A_3 $, where
\begin{align*}
	A_1 
	&:= 
	\E \int_0^2 \int_{\R} \devm(t,x-\textbb) \big( \devm(t,x-\textbb) - \dev(t,x) \big) \ \d x \d t,
\\
	A_2
	&:=
	\E \int_0^2 \int_{\R}  \devm^2(t,x) - \devm^2(t,x-\textbb) \ \d x \d t,	
\\
	A_3
	&:=
	\E \int_0^2 \int_{\R} \big( \devm(t,x-\textbb) - \devm(t,x) \big) \dev(t,x) \ \d x \d t.	
\end{align*}
For $ A_1 $, the change of variables $ x=\alpha\ell(t) + \textbb = \alpha \ell(t) + \scl^{-1/2} \bb(t) $ reveals that $ A_1 $ is equal to the left hand side of \eqref{e.deduced}.
Hence $ A_1 \leq \scl^{-1} \frac{16}{15}\log\sqrt{4\pi} $.
The term $ A_2 $ does not depend on $ \dev $, and it is readily checked from \eqref{e.devm} that $ \lim_{\scl\to\infty} |A_2| =0 $.
As for $ A_3 $, the Cauchy--Schwarz inequality gives $ |A_3|\leq A_{31}^{1/2} \normL{ \dev } $,
where $ A_{31} := \E \int ( \devm(t,x-\textbb) - \devm(t,x) )^2 \d t \d x $.
The term $ A_{31} $ does not depend on $ \dev $, and it is readily checked from \eqref{e.devm} that $ \lim_{\scl\to\infty} |A_{31}| =0 $.
Adopt the notation $ o_\scl(1) $ for a generic quantity that depends only on $ \scl $ such that $ \lim_{\scl\to\infty}|o_\scl(1)| = 0 $.
Collecting the preceding results on $ A_1 $, $ A_2 $, and $ A_3 $ now gives
\begin{align}
	\label{e.quantitative}
	\ip{ \devm,\devm-\dev } \leq o_\scl(1) ( 1 + \normL{\dev} ).
\end{align}

Since $ \normL{\dev}^2 = \normL{\devm}^2 + \normL{\dev-\devm}^2 - 2\ip{ \devm,\devm-\dev } $,
the bound \eqref{e.quantitative} implies
$	\normL{ \devm }^2 \leq (1+o_\scl(1))\normL{\dev}^2 + o_\scl(1).$
This inequality holds for all $ \dev\in\Lsp^2 $ with $ \hfn_\scl(\dev;0,2) \leq -1 $, and $  o_\scl(1)\to 0 $ does not depend on $ \dev $.
The desired result hence follows:
\begin{align*}
\liminf_{\scl\to\infty}
	\big( \inf\big\{ \tfrac12 \normL{ \dev }^2  : \, \hfn_\scl(\dev;2,0) \leq -1 \big\} \big)
	\geq
	\tfrac12 \normL{ \devm }^2 = \tfrac{4}{15\pi}.	
\end{align*}

\bibliographystyle{alphaabbr}
\bibliography{short-time-kpz}

\newcommand{\etalchar}[1]{$^{#1}$}
\begin{thebibliography}{HLDM{\etalchar{+}}18}

\bibitem[BDM08]{budhiraja08}
A.~Budhiraja, P.~Dupuis, and V.~Maroulas.
\newblock Large deviations for infinite dimensional stochastic dynamical
  systems.
\newblock {\em Ann Probab}, pages 1390--1420, 2008.

\bibitem[BDSG{\etalchar{+}}15]{bertini15}
L.~Bertini, A.~De~Sole, D.~Gabrielli, G.~Jona-Lasinio, and C.~Landim.
\newblock Macroscopic fluctuation theory.
\newblock {\em Rev Modern Phys}, 87(2):593, 2015.

\bibitem[BGS19]{basu19}
R.~Basu, S.~Ganguly, and A.~Sly.
\newblock Delocalization of polymers in lower tail large deviation.
\newblock {\em Commun Math Phys}, 370(3):781--806, 2019.

\bibitem[CC19]{cafasso19}
M.~Cafasso and T.~Claeys.
\newblock A {R}iemann{-H}ilbert approach to the lower tail of the {KPZ}
  equation.
\newblock {\em arXiv:1910.02493}, 2019.

\bibitem[CD19]{cerrai19}
S.~Cerrai and A.~Debussche.
\newblock Large deviations for the two-dimensional stochastic
  {N}avier{--S}tokes equation with vanishing noise correlation.
\newblock In {\em Annales de l'Institut Henri Poincar{\'e}, Probabilit{\'e}s et
  Statistiques}, volume~55, pages 211--236. Institut Henri Poincar{\'e}, 2019.

\bibitem[CG20a]{corwin20general}
I.~Corwin and P.~Ghosal.
\newblock {KPZ} equation tails for general initial data.
\newblock {\em Electron J Probab}, 25, 2020.

\bibitem[CG20b]{corwin20lower}
I.~Corwin and P.~Ghosal.
\newblock Lower tail of the {KPZ} equation.
\newblock {\em Duke Math J}, 169(7):1329--1395, 2020.

\bibitem[CGK{\etalchar{+}}18]{corwin18}
I.~Corwin, P.~Ghosal, A.~Krajenbrink, P.~Le~Doussal, and L.-C. Tsai.
\newblock Coulomb-gas electrostatics controls large fluctuations of the
  {K}ardar{-P}arisi{-Z}hang equation.
\newblock {\em Phys Rev Lett}, 121(6):060201, 2018.

\bibitem[CHN16]{chen16}
L.~Chen, Y.~Hu, and D.~Nualart.
\newblock Regularity and strict positivity of densities for the nonlinear
  stochastic heat equation.
\newblock {\em arXiv:1611.03909}, 2016.

\bibitem[CM97]{chenal97}
F.~Chenal and A.~Millet.
\newblock Uniform large deviations for parabolic {SPDE}s and applications.
\newblock {\em Stochastic Process Appl}, 72(2):161--186, 1997.

\bibitem[Cor12]{corwin2012kardar}
I.~Corwin.
\newblock The {K}ardar{--P}arisi{--Z}hang equation and universality class.
\newblock {\em Random Matrices: Theory Appl}, 1(01):1130001, 2012.

\bibitem[Cor18]{Cor18}
I.~Corwin.
\newblock Exactly solving the {KPZ} equation.
\newblock {\em arXiv:1804.05721}, 2018.

\bibitem[CS19]{corwin2019}
I.~Corwin and H.~Shen.
\newblock Some recent progress in singular stochastic {PDE}s.
\newblock {\em Bull Amer Math Soc}, 57:409--454, 2019.

\bibitem[CW17]{chandra2017stochastic}
A.~Chandra and H.~Weber.
\newblock {Stochastic PDEs, regularity structures, and interacting particle
  systems}.
\newblock In {\em Annales de la facult{\'e} des sciences de Toulouse
  Math{\'e}matiques}, volume~26, pages 847--909, 2017.

\bibitem[DL98]{derrida98}
B.~Derrida and J.~L. Lebowitz.
\newblock Exact large deviation function in the asymmetric exclusion process.
\newblock {\em Phys Rev Lett}, 80(2):209, 1998.

\bibitem[DS01]{DS01}
J.-D. Deuschel and D.~W. Stroock.
\newblock {\em Large deviations}, volume 342.
\newblock American Mathematical Soc., 2001.

\bibitem[DT19]{das19}
S.~Das and L.-C. Tsai.
\newblock Fractional moments of the stochastic heat equation.
\newblock {\em arXiv:1910.09271}, 2019.

\bibitem[DZ94]{DZ94}
A.~Dembo and O.~Zeitouni.
\newblock Large deviations techniques and applications.
\newblock {\em Applications of Mathematics (New York)}, 38, 1994.

\bibitem[EK04]{elgart04}
V.~Elgart and A.~Kamenev.
\newblock Rare event statistics in reaction-diffusion systems.
\newblock {\em Phys Rev E}, 70(4):041106, 2004.

\bibitem[Eva98]{evans1998partial}
L.~C. Evans.
\newblock Partial differential equations.
\newblock {\em Graduate studies in mathematics}, 19(2), 1998.

\bibitem[FGV01]{falkovich01}
G.~Falkovich, K.~Gawedzki, and M.~Vergassola.
\newblock Particles and fields in fluid turbulence.
\newblock {\em Rev Modern Phys}, 73(4):913, 2001.

\bibitem[FKLM96]{falkovich96}
G.~Falkovich, I.~Kolokolov, V.~Lebedev, and A.~Migdal.
\newblock Instantons and intermittency.
\newblock {\em Physical Review E}, 54(5):4896, 1996.

\bibitem[Flo14]{flores14}
G.~R.~M. Flores.
\newblock On the (strict) positivity of solutions of the stochastic heat
  equation.
\newblock {\em Ann Probab}, 42(4):1635--1643, 2014.

\bibitem[Fog98]{fogedby98}
H.~C. Fogedby.
\newblock Soliton approach to the noisy burgers equation: Steepest descent
  method.
\newblock {\em Phys Rev E}, 57(5):4943, 1998.

\bibitem[FS10]{ferrari2010random}
P.~L. Ferrari and H.~Spohn.
\newblock Random growth models.
\newblock {\em arXiv:1003.0881}, 2010.

\bibitem[GGS15]{grafke15}
T.~Grafke, R.~Grauer, and T.~Sch\"{a}fer.
\newblock The instanton method and its numerical implementation in fluid
  mechanics.
\newblock {\em J Phys A: Math Theor}, 48(33):333001, 2015.

\bibitem[GIP15]{gubinelli2015paracontrolled}
M.~Gubinelli, P.~Imkeller, and N.~Perkowski.
\newblock Paracontrolled distributions and singular {PDE}s.
\newblock In {\em Forum of Mathematics, Pi}, volume~3. Cambridge University
  Press, 2015.

\bibitem[GL20]{ghosal20}
P.~Ghosal and Y.~Lin.
\newblock {Lyapunov exponents of the SHE for general initial data}.
\newblock {\em arXiv:2007.06505}, 2020.

\bibitem[Hai14]{hairer14}
M.~Hairer.
\newblock A theory of regularity structures.
\newblock {\em Invent Math}, 198(2):269--504, 2014.

\bibitem[HL66]{halperin66}
B.~Halperin and M.~Lax.
\newblock Impurity-band tails in the high-density limit. i. minimum counting
  methods.
\newblock {\em Phys Rev}, 148(2):722, 1966.

\bibitem[HL18]{hu2018asymptotics}
Y.~Hu and K.~L{\^e}.
\newblock Asymptotics of the density of parabolic {A}nderson random fields.
\newblock {\em arXiv:1801.03386}, 2018.

\bibitem[HLDM{\etalchar{+}}18]{hartmann2018high}
A.~K. Hartmann, P.~Le~Doussal, S.~N. Majumdar, A.~Rosso, and G.~Schehr.
\newblock High-precision simulation of the height distribution for the {KPZ}
  equation.
\newblock {\em EPL (Europhysics Letters)}, 121(6):67004, 2018.

\bibitem[HMS19]{hartmann2019optimal}
A.~K. Hartmann, B.~Meerson, and P.~Sasorov.
\newblock Optimal paths of nonequilibrium stochastic fields: The
  {K}ardar{-P}arisi{-Z}hang interface as a test case.
\newblock {\em Physical Review Research}, 1(3):032043, 2019.

\bibitem[HW15]{hairer15}
M.~Hairer and H.~Weber.
\newblock Large deviations for white-noise driven, nonlinear stochastic {PDE}s
  in two and three dimensions.
\newblock {\em Annales de la Facult{\'e} des sciences de Toulouse:
  Math{\'e}matiques}, 24(1):55--92, 2015.

\bibitem[KK07]{kolokolov07}
I.~Kolokolov and S.~Korshunov.
\newblock Optimal fluctuation approach to a directed polymer in a random
  medium.
\newblock {\em Phys Rev B}, 75(14):140201, 2007.

\bibitem[KK09]{kolokolov09}
I.~Kolokolov and S.~Korshunov.
\newblock Explicit solution of the optimal fluctuation problem for an elastic
  string in a random medium.
\newblock {\em Phys Rev E}, 80(3):031107, 2009.

\bibitem[KLD17]{krajenbrink17short}
A.~Krajenbrink and P.~Le~Doussal.
\newblock Exact short-time height distribution in the one-dimensional
  {K}ardar{--P}arisi{--Z}hang equation with {B}rownian initial condition.
\newblock {\em Phys Rev E}, 96(2):020102, 2017.

\bibitem[KLD18a]{krajenbrink2018large}
A.~Krajenbrink and P.~Le~Doussal.
\newblock Large fluctuations of the {KPZ} equation in a half-space.
\newblock {\em SciPost Phys}, 5:032, 2018.

\bibitem[KLD18b]{krajenbrink2018simple}
A.~Krajenbrink and P.~Le~Doussal.
\newblock {Simple derivation of the $(-\lambda H)^{5/2}$ tail for the 1{D}
  {KPZ} equation}.
\newblock {\em J Stat Mech Theory Exp}, 2018(6):063210, 2018.

\bibitem[KLD19]{krajenbrink2019linear}
A.~Krajenbrink and P.~Le~Doussal.
\newblock Linear statistics and pushed {C}oulomb gas at the edge of
  $\beta$-random matrices: Four paths to large deviations.
\newblock {\em EPL (Europhysics Letters)}, 125(2):20009, 2019.

\bibitem[KLDP18]{krajenbrink18}
A.~Krajenbrink, P.~Le~Doussal, and S.~Prolhac.
\newblock Systematic time expansion for the {K}ardar--{P}arisi--{Z}hang
  equation, linear statistics of the {GUE} at the edge and trapped fermions.
\newblock {\em Nucl Phys B}, 936:239--305, 2018.

\bibitem[KMS16]{kamenev16}
A.~Kamenev, B.~Meerson, and P.~V. Sasorov.
\newblock Short-time height distribution in the one-dimensional
  {K}ardar{--P}arisi{--Z}hang equation: Starting from a parabola.
\newblock {\em Phys Rev E}, 94(3):032108, 2016.

\bibitem[KPZ86]{kardar1986dynamic}
M.~Kardar, G.~Parisi, and Y.-C. Zhang.
\newblock Dynamic scaling of growing interfaces.
\newblock {\em Phys Rev Lett}, 56(9):889, 1986.

\bibitem[Kra19]{krajenbrink2019beyond}
A.~Krajenbrink.
\newblock {\em Beyond the typical fluctuations: a journey to the large
  deviations in the {K}ardar{-P}arisi{-Z}hang growth model}.
\newblock PhD thesis, PSL Research University, 2019.

\bibitem[LD19]{ledoussal19}
P.~Le~Doussal.
\newblock {Large deviations for the KPZ equation from the KP equation}.
\newblock {\em arXiv preprint arXiv:1910.03671}, 2019.

\bibitem[LDMRS16]{ledoussal16short}
P.~Le~Doussal, S.~N. Majumdar, A.~Rosso, and G.~Schehr.
\newblock Exact short-time height distribution in the one-dimensional
  {K}ardar{--P}arisi{--Z}hang equation and edge fermions at high temperature.
\newblock {\em Phys Rev Lett}, 117(7):070403, 2016.

\bibitem[LDMS16]{ledoussal16long}
P.~Le~Doussal, S.~N. Majumdar, and G.~Schehr.
\newblock Large deviations for the height in 1{D} {K}ardar-{P}arisi-{Z}hang
  growth at late times.
\newblock {\em EPL (Europhysics Letters)}, 113(6):60004, 2016.

\bibitem[Led96]{ledoux96}
M.~Ledoux.
\newblock {Isoperimetry and Gaussian analysis}.
\newblock In {\em Lectures on probability theory and statistics}, pages
  165--294. Springer, 1996.

\bibitem[Lif68]{lifshitz68}
I.~Lifshitz.
\newblock Theory of fluctuating levels in disordered systems.
\newblock {\em Sov Phys JETP}, 26(462):012110--9, 1968.

\bibitem[MKV16]{meerson16}
B.~Meerson, E.~Katzav, and A.~Vilenkin.
\newblock Large deviations of surface height in the {K}ardar{--P}arisi{--Z}hang
  equation.
\newblock {\em Phys Rev Lett}, 116(7):070601, 2016.

\bibitem[MN08]{mueller2008regularity}
C.~Mueller and D.~Nualart.
\newblock Regularity of the density for the stochastic heat equation.
\newblock {\em Electron J Probab}, 13:2248--2258, 2008.

\bibitem[MS11]{meerson11}
B.~Meerson and P.~V. Sasorov.
\newblock Negative velocity fluctuations of pulled reaction fronts.
\newblock {\em Phys Rev E}, 84(3):030101, 2011.

\bibitem[MS17]{meerson2017height}
B.~Meerson and J.~Schmidt.
\newblock Height distribution tails in the {K}ardar{--P}arisi{--Z}hang equation
  with brownian initial conditions.
\newblock {\em J Stat Mech Theory Exp}, 2017(10):103207, 2017.

\bibitem[Mue91]{mueller91}
C.~Mueller.
\newblock On the support of solutions to the heat equation with noise.
\newblock {\em Stochastics: An International Journal of Probability and
  Stochastic Processes}, 37(4):225--245, 1991.

\bibitem[MV18]{meerson2018large}
B.~Meerson and A.~Vilenkin.
\newblock Large fluctuations of a {K}ardar{-P}arisi{-Z}hang interface on a half
  line.
\newblock {\em Physical Review E}, 98(3):032145, 2018.

\bibitem[Nua06]{Nua06}
D.~Nualart.
\newblock {\em The {M}alliavin calculus and related topics}, volume 1995.
\newblock Springer, 2006.

\bibitem[QS15]{quastel2015one}
J.~Quastel and H.~Spohn.
\newblock The one-dimensional {KPZ} equation and its universality class.
\newblock {\em J Stat Phys}, 160(4):965--984, 2015.

\bibitem[Qua11]{quastel2011introduction}
J.~Quastel.
\newblock Introduction to {KPZ}.
\newblock {\em Current developments in mathematics}, 2011(1), 2011.

\bibitem[SMP17]{sasorov17}
P.~Sasorov, B.~Meerson, and S.~Prolhac.
\newblock Large deviations of surface height in the 1+1-dimensional
  {K}ardar{--P}arisi{--Z}hang equation: exact long-time results for ${\lambda}
  h< 0$.
\newblock {\em J Stat Mech Theory Exp}, 2017(6):063203, 2017.

\bibitem[Tsa18]{tsai18}
L.-C. Tsai.
\newblock Exact lower tail large deviations of the {KPZ} equation.
\newblock {\em arXiv:1809.03410}, 2018.

\bibitem[ZL66]{zittartz66}
J.~Zittartz and J.~Langer.
\newblock Theory of bound states in a random potential.
\newblock {\em Phys Rev}, 148(2):741, 1966.

\end{thebibliography}

\end{document}